\documentclass[11pt]{article}
\usepackage{enumerate}
\usepackage{amsthm}
\usepackage{amsmath}
\usepackage{amssymb}
\usepackage{bbold,bbm}
\usepackage{hyperref}
\hypersetup{
    colorlinks,
    citecolor=red,
    filecolor=black,
    linkcolor=red,
    urlcolor=black
}
\newcommand{\footremember}[2]{%
    \footnote{#2}
    \newcounter{#1}
    \setcounter{#1}{\value{footnote}}%
}

\title{Stochastic integration with respect to \\ canonical $\alpha$-stable cylindrical Lévy processes}
\author{
  Gergely Bodó \\
  \and
  Markus Riedle\footremember{trailer}{markus.riedle@kcl.ac.uk}\footremember{alley2}{2nd affiliation: Institute of Mathematical Stochastics, Faculty of Mathematics,	TU Dresden,   01062 Dresden, Germany}
}
\date{%
    Department of Mathematics\\%
    King's College London\\%
    London WC2R 2LS\\%
    United Kingdom\\[2ex]%
    \today
}

\newtheorem{theorem}{Theorem}[section]
\newtheorem{corollary}[theorem]{Corollary}
\newtheorem{lemma}[theorem]{Lemma}
\newtheorem{definition}[theorem]{Definition}
\newtheorem{remark}[theorem]{Remark}
\newtheorem{proposition}[theorem]{Proposition}
\newcommand{\norm}[1]{\left\lVert#1\right\rVert}
\newcommand{\normm}[1]{{\left\vert\kern-0.25ex\left\vert\kern-0.25ex\left\vert #1 
    \right\vert\kern-0.25ex\right\vert\kern-0.25ex\right\vert}}
\numberwithin{equation}{section}

\DeclareMathOperator{\Borel}{{\mathfrak B}}
\renewcommand{\d}{{\mathrm d}}

\newcommand{\scapro}[2]{\langle #1,#2\rangle}  
\DeclareMathOperator{\R}{{\mathbb R}}

\renewcommand{\Phi}{F}
\newcommand{\1}{\mathbbm{1}} 
\renewcommand{\phi}{\varphi}

\newcommand{\itemEq}[1]{%
	\begingroup%
	\setlength{\abovedisplayskip}{0pt}%
	\setlength{\belowdisplayskip}{0pt}%
	\parbox[c]{\linewidth}{\begin{flalign}#1&&\end{flalign}}%
	\endgroup}

\providecommand{\keywords}
{
  \textbf{\textit{Keywords and phrases:}}
}
\providecommand{\amssubj}
{
  \textbf{\textit{AMS 2010 subject classification:}}
}

\begin{document}

\maketitle

\begin{abstract} 
In this work,  we introduce a theory of stochastic integration with respect to symmetric $\alpha$-stable cylindrical L\'evy processes. Since $\alpha$-stable cylindrical L\'evy processes do not enjoy a semi-martingale decomposition, our approach is based on a decoupling inequality for the tangent sequence of the Radonified increments. This approach enables us to characterise  the largest space of predictable Hilbert-Schmidt operator-valued processes which are integrable with respect to an $\alpha$-stable cylindrical L\'evy process as the collection of all predictable processes with paths in the Bochner space $L^\alpha$. We demonstrate the power and robustness of the developed theory by establishing a dominated convergence result allowing the interchange of the stochastic integral and limit. 
\end{abstract}

\begin{flushleft}
\amssubj{60H05, 60G20, 60G52, 28C20}

\end{flushleft}

\section{Introduction}

Symmetric $\alpha$-stable distributions are popular for modelling random perturbations in the Euclidean space; see e.g. Samorodnitsky and Taqqu \cite{Samorodnitsky-Taqqu}.  This 
is because they are natural, discontinuous and non-Gaussian generalisations of Brownian motions  meeting various empirical requests, such as heavy tails, self-similarity and infinite variance. As perturbations of infinite dimensional systems, such as partial differential equations, models with stable distributions can usually only be found realised as random fields but not as stochastic processes in a Hilbert or Banach space.  This fact is not surprising since a random noise with a symmetric $\alpha$-stable distribution does not exist as an ordinary  Hilbert space-valued process but only in the generalised sense of  Gel'fand and Vilenkin \cite{Gelfand_Vilenkin} or  Segal \cite{Segal} as cylindrical processes. This is analogous to the standard Brownian motion in an infinite dimensional Hilbert space which only exists as a cylindrical process. 

The purpose of this work is to provide a comprehensive and robust theory of stochastic integration for predictable processes with respect to symmetric $\alpha$-stable cylindrical processes, which will lay the foundation for the application of these processes as a model of random perturbations of infinite dimensional systems. We do not only introduce the stochastic integral but also characterise the largest set of predictable Hilbert-Schmidt-valued processes which are integrable with respect to a symmetric $\alpha$-stable cylindrical process.
Including predictable integrands in the theory of stochastic integration is important, as solutions of stochastic partial differential equation driven by cylindrical L\'evy processes do not necessarily have c\`adl\`ag trajectories; see e.g. Brze\'zniak et al.\ \cite{Brzezniak_et_al}. 
We demonstrate the power of the developed integration theory by establishing a dominated convergence result allowing the interchange of the stochastic integral and limit. 

The classical approach to stochastic integration in finite-dimensional spaces is based on the semi-martingale decomposition; see Dellacherie and Meyer \cite{Dellacherie-Meyer-B}. The reverse approach, starting with good integrators, is introduced in  Protter \cite{protter_2005} or in Kurtz and Protter \cite{Kurtz-Protter} for an infinite dimensional setting; however, extending the space from adapted, c\`agl\`ad  to arbitrary predictable integrands in Protter \cite{protter_2005} also depends on the semi-martingale decomposition. A completely different approach is introduced in Bichteler \cite{bichteler_2002}, where the construction of the stochastic integral for real-valued integrands and integrators mimics that of the Daniell integral in calculus. Another approach defines the stochastic integral as a vector-valued random measure; 
see e.g. Métivier and Pellaumail \cite{metivier_michel_pellaumail_1980} or Rao \cite{Rao-2012}. A decoupling inequality for tangent sequences is the foundation for stochastic integration developed by Kwapie\'n and Woyczy\'nski in \cite{kwapien_woyczynski_1986} and \cite{K_W_1987}; their approach also allows for the characterisation of the largest 
space of predictable integrands as a randomised Musielak-Orlicz space. This approach is extended to Hilbert space-valued semi-martingales in Nowak \cite{nowak_2003}.  

Turning to stochastic integration with respect to  cylindrical processes, an extensive theory has been developed 
for stochastic integrals with respect to cylindrical Brownian motions, which has recently been  extended even to UMD Banach spaces; see van Neerven \cite{Jan-annals}. Surprisingly, stochastic integration with respect to other cylindrical processes than cylindrical Brownian motion is much less considered.  In fact, only  with respect to cylindrical martingales a stochastic integration theory is developed either by following a Dol\'eans measure approach by M\'etivier and Pellaumail in \cite{MetivierPellcylindrical} and \cite{metivier_michel_pellaumail_1980}, or by constructing a family of reproducing kernel Hilbert spaces in  Mikulevi\v{c}ius and Rozovski\v{\i} in \cite{MikRoz98} and \cite{MikRoz99}.
For the special case of a cylindrical L{\'e}vy process with finite weak second moments, an It{\^o} approach to  stochastic integration is developed in Riedle \cite{Riedle14}. The reason for the restriction to cylindrical martingales is due to the fact that cylindrical semi-martingales do not enjoy a semi-martingale decomposition, and thus the classical approach to stochastic integration, mentioned above, is not applicable. This obstacle is overcome in Jakubowski and Riedle \cite{jakubowski_riedle_2017}, in which a stochastic integral for adapted, c\`agl\`ad integrands and arbitrary cylindrical L\'evy processes is introduced. The approach is based on a decoupling inequality for tangent sequences, but seems neither to allow an extension to include predictable integrands nor to establish some powerful limit theorems. 

The approach in the current work also relies on a decoupling inequality for tangent sequences, but we characterise convergence of the  stochastic integral in the semi-martingale topology to be equivalent to convergence in probability in the space of random variables with values in a Bochner space. This approach originates from Kwapie\'n and Woyczy\'nski \cite{kwapien_1992}, in which the equivalent topology is defined on the space of random variables with values in a certain Musielak-Orlicz space. This equivalent description of a topology  for the convergence of the  stochastic integral allows the characterisation of the largest space of predictable integrands. 

The paper is organised as follows: we summarise some preliminaries on genuine and cylindrical L\'evy processes in Section 2. The theory of stochastic integration for deterministic integrands is introduced in Section 3. The following Section 4 is devoted to providing an approximation result for the predictable integrands under consideration. In Section 5, we explicitly construct the tangent sequence of the Radonified increments, which is 
fundamental for our approach. In the final Section 6, we state our main result on the equivalent description of the largest space of predictable integrands as a Bochner space $L^\alpha$, and provide its proof. We finish the  section with a result on dominated convergence.

\section{Preliminaries}

Let $G$ and $H$ be separable Hilbert spaces with inner products $\langle \cdot,\cdot\rangle$ and corresponding norms $\norm{\cdot}$. Let  $(a_k)_{k\in \mathbb{N}}$ and $(b_k)_{k\in \mathbb{N}}$ be  orthonormal bases of $G$ and $H$, respectively. We identify the dual of a Hilbert space by the space itself. The Borel $\sigma$-algebra of $H$ is denoted by $\Borel(H)$ and the open unit ball by $B_H:=\{h\in H: \norm{h}<1\}$.\par
The Banach space of bounded linear operators from $G$ to $H$ will be denoted by $L(G,H)$ with the operator norm $\norm{\cdot}_{G \rightarrow H}$. Its subspace $L_2(G,H)$ of Hilbert-Schmidt operators is endowed with the norm
$\norm{\Phi}_{\rm HS}^2:=\sum_{k=1}^{\infty}\norm{\Phi a_k}^2$ for $\Phi \in L_2(G,H)$.
We denote by $L_{\rm Leb}^{\alpha}([0,T],L_2(G,H))$ the space of Borel measurable functions $f:[0,T]\rightarrow L_2(G,H)$ that satisfy the condition $\int_0^T \norm{f(t)}_{\rm HS}^{\alpha}\, {\rm d}t < \infty$. For $\alpha\ge 1$, this is a Banach space
if equipped with the usual norm $\norm{f}_{L^\alpha}^\alpha:=\int_0^T \norm{f(t)}_{\rm HS}^{\alpha} \, {\rm d}t $, and for $0<\alpha<1$ it is  a metric space under the translation invariant metric $d(f,g)=\int_0^T \norm{f(t)-g(t)}_{\rm HS}^{\alpha} \, {\rm d}t $. For ease of notation, we also use the notation $\norm{f}_{L^\alpha}$ to denote the metric $d(0,f)$ for $0<\alpha<1$. \par
Let $(\Omega,\Sigma,P)$ be a complete probability space. We will denote by $L_P^0(\Omega,H)$ the space of equivalence classes of measurable functions $X\colon \Omega \rightarrow H$, equipped with the topology of convergence in probability.\par
Let $S$ be a subset of $G$. For each $n \in \mathbb{N}$, elements $g_1,...,g_n \in S$ and Borel set $A \in \Borel(\mathbb{R}^n)$, we define
\[C(g_1,...,g_n;A):=\{g \in G: (\langle g,g_1 \rangle,...,\langle g,g_n \rangle)\in A\}.\]
Such sets are called cylindrical sets with respect to $A$. The set of all these cylindrical sets is denoted by  $\mathcal{Z}(G,S)$,  and  it is a $\sigma$-algebra if $S$ is finite and otherwise an algebra. We write 
$\mathcal{Z}(G)$ for $\mathcal{Z}(G,G)$\par
A set function $\mu: \mathcal{Z}(G)\rightarrow [0, \infty]$ is called a cylindrical measure on $\mathcal{Z}(G)$ if for each finite dimensional subset $S \subseteq G$, the restriction of $\mu$ to the $\sigma$-algebra $\mathcal{Z}(G,S)$ is a $\sigma$-additive measure. A cylindrical measure is said to be a cylindrical probability measure if $\mu(G)=1$.\par
A cylindrical random variable $X$ in $G$ is a linear and continuous mapping $X\colon G \rightarrow L_{P}^0(\Omega,\mathbb{R})$. It defines a cylindrical probability measure $\mu_X$ by 
\begin{align*}
	\mu_X\colon  \mathcal{Z}(G) \to [0,1],\qquad
	\mu_X(Z)=P\big( (Xg_1,\dots, Xg_n)\in A\big)
\end{align*}
for cylindrical sets $Z=C(g_1, ... , g_n; A)$. The cylindrical probability measure $\mu_X$ is called the {\em cylindrical distribution} of $X$.
We define the characteristic function of the cylindrical random variable $X$ by
\[\varphi_X\colon G \rightarrow \mathbb{C}, \qquad\varphi_X(g)=E\big[e^{iXg}\big].\]
Let $T\colon G\to H$ be a linear and continuous operator. By defining
\begin{align*}
	TX\colon H\to L_{P}^0(\Omega,\mathbb{R}), \qquad (TX)h=X(T^\ast h)
\end{align*}
we obtain a cylindrical random variable on $H$. In the special case when $T$ is a Hilbert-Schmidt operator and hence $0$-Radonifying by \cite[Th.\ VI.5.2]{vakhania_1981}, it follows from  \cite[Pr.\ VI.5.3]{vakhania_1981} that the cylindrical random variable $TX$ is induced by a genuine random variable $Y\colon \Omega\to H$, that is 
$(TX)h =\scapro{Y}{h}$  for all $h\in H$. The following result shows that the inducing random variable $Y$ continuously depends on the Hilbert-Schmidt operator. 
\begin{lemma}\label{le.conv_of_radonif_rv}
	Let $X$ be a cylindrical random variable and $(F_n)_{n \in \mathbb{N}}$ a sequence in $L_2(G,H)$ converging to $F$ in $\norm{\cdot}_{\rm HS}$. Then $(F_nX)_{n \in \mathbb{N}}$ converges to $FX$ in $L_P^0(\Omega,H)$.
\end{lemma}
\begin{proof} Let $\mu_X$ denote the cylindrical distribution of $X$. 
As the sequence  $(F_n)_{n \in \mathbb{N}}$ is compact in $L_2(G,H)$, the collection of measures $\{\mu_X \circ F_n^{-1}: n \in \mathbb{N}\}$ is relatively compact in the space of probability measures on $\Borel(H)$; see \cite[Pr.\ 5.3]{jakubowski_riedle_2017}. Continuity of $X$ implies  for all $h \in H$ that
\begin{equation*}
    \lim_{n \rightarrow \infty}\langle F_n X, h \rangle= \lim_{n \rightarrow \infty} X(F_n^* h)= X(F^* h)=\langle FX, h \rangle\qquad\text{in }L^0_P(\Omega,\R).
\end{equation*}
Together with relative compactness, this implies that $(F_n X)_{n \in \mathbb{N}}$ converges to $FX$ in $L_P^0(\Omega,H)$; see e.g.\ \cite[Le.\ 2.4]{jakubowski_1988}. 
\end{proof}
A family $(L(t):t\geq 0)$ of cylindrical random variables $L(t)\colon G \rightarrow L_{P}^0(\Omega,\mathbb{R})$ is called a cylindrical Lévy process if for each $n \in \mathbb{N}$ and $g_1,...,g_n \in G$, the stochastic process
\[\big(\big(L(t)g_1,...,L(t)g_n\big): t \geq 0\big)\]
is a Lévy process in $\mathbb{R}^n$. The filtration generated by $(L(t):t\geq 0)$ is defined by
\[\mathcal{F}_t :=\sigma(\{L(s)g\colon g\in G, s\in [0,t]\}) \quad \text{for all}\; t\geq 0.\] 
Denote by $\mathcal{Z}_*(G)$ the collection
\[\big\{\{g \in G: (\langle g,g_1 \rangle,...,\langle g,g_n \rangle)\in B\}:n \in \mathbb{N},g_1,...,g_n \in G,B\in \Borel{B}(\mathbb{R}^n\setminus\{0\})\big\}\]
of cylindrical sets, which  forms an algebra of subsets of $G$. For fixed $g_1,...,g_n \in G$, let $\lambda_{g_1,\dots, g_n}$ be the L\'evy measure of $\big((L(t)g_1,...,L(t)g_n ): t \geq 0\big)$. Define a function 
$\lambda\colon  \mathcal{Z}_*(G) \rightarrow [0,\infty]$ by 
\begin{align*}
\lambda(C):=\lambda_{g_1,\dots, g_n}(B)\quad\text{for }
C=\{g\in G:\, (\langle g,g_1 \rangle,...,\langle g,g_n \rangle)\in B\},
\end{align*}
for $B\in\Borel(\R^n)$. It is shown in \cite{applebaum_riedle_2010} that $\lambda$ is well defined. The set function $\lambda$ is called the cylindrical L\'evy measure of $L$. 

In this paper, we restrict our attention to canonical $\alpha$-stable cylindrical Lévy processes. These are 
cylindrical Lévy processes with characteristic function $\phi_{L(t)}(g)=\exp(-t\norm{g}^\alpha)$ for each $t\ge 0$ and $g\in G$.
By \cite[Le.\ 2.4]{riedle_2018}, the cylindrical L\'evy measure $\lambda$ of the canonical $\alpha$-stable cylindrical Lévy processes satisfies the spectral 
representation
\begin{equation}\label{spectral_represenation_of_stable_levy}
    \lambda \circ \pi_{a_1,...,a_n}^{-1}(B)=\frac{\alpha}{c_{\alpha}}\int_{S_{\mathbb{R}^n}}\nu_n({\rm d}x) \int_0^{\infty} \mathbb{1}_B(r x)\frac{1}{r^{1+\alpha}}\,{\rm d}r \quad \text{for}\; B \in \Borel({\mathbb{R}^n}),
\end{equation}
where $S_{\mathbb{R}^n}:=\{\beta \in \mathbb{R}^n: \vert \beta \vert = 1\}$, $c_\alpha>0$ is a constant dependent only on $\alpha$ and $\nu_n$ denotes a uniform distribution on the sphere $S_{{\mathbb R}^n}$.

Infinitely divisible measures on a Hilbert space $H$ can be defined as in the Euclidean space; see \cite{parthasarathy_1967}. 
As in finite dimensions, infinitely divisible distributions are completely characterised by a triplet $(a,\rho,\lambda)$, where $a\in H$, the mapping $\rho\colon H\to H$ is nuclear and non-negative, and the Lévy measure  $\lambda$  is a $\sigma$-finite measure on $\Borel(H)$ satisfying
$\int_H \big(\norm{h}^2\wedge 1\big)\, \lambda({\rm d}h)<\infty$. Given any $\delta>0$ and Lévy measure $\lambda$ on $\Borel(H)$, we say that $\delta \in C(\lambda)$ if $\lambda(\{h \in H: \norm{h}=\delta\})=0$. A sequence of infinitely divisible measures $\mu_n=(a_n,\rho_n,\lambda_n)$ with associated sequence $(T_n)_{n \in \mathbb{N}}$ of $S$-operators $T_n\colon H\to H$, which are defined by
\begin{align*}
    \langle T_nh_1,h_2\rangle=\langle \rho_n h_1,h_2\rangle + \int_{\norm{h_1}\leq 1}\langle h_1,u \rangle \langle h_2, u\rangle \, \lambda_n({\rm d}u)\qquad\text{for all }h_1,h_2\in H,
\end{align*}
converges weakly to an infinitely divisible measure $\mu=(a,\rho,\lambda)$ if and only if the following conditions  hold:
\begin{enumerate}[\hspace{0.5cm}]
	\item \itemEq{(1)\;  \displaystyle a=\lim_{\substack{\delta \downarrow 1\\ \delta \in C(\lambda)}}\lim_{n \rightarrow \infty}\bigg(a_n+\int_{1<\norm{h}\leq\delta}h\;\lambda_n(dh)\bigg);}
	\item \itemEq{(2)\;   \displaystyle \lim_{\delta\downarrow 0} \limsup_{n\rightarrow \infty}
	\int_{\norm{h}\leq \delta} \langle h,u \rangle^2 \, \lambda_n({\rm d}u)+(\rho_n h,h)=(\rho h,h) \text{ for all }h\in H; 	
	 \label{small_jumps_conv} }
	\item \itemEq{\displaystyle (3)\;  \lambda_n \to \lambda \text{ weakly outside of every closed neighbourhood of the origin;}\label{conv_of_levy_measure}}
	\item \itemEq{(4)\; \text{$(T_n)_{n \in \mathbb{N}}$ is compact in the space of nuclear operators.}} \label{comp_S_op}
\end{enumerate}
The necessity of these conditions can be found in \cite[Pr.\ 5.7.4]{linde_1987} and their sufficiency is an adaption of \cite[Th.\ 5.5]{parthasarathy_1967} to the case of a discontinuous truncation function. 

It is well known, even in finite dimensions, that the small jumps of the converging sequence $\mu_n$ may contribute to the Gaussian part of the limit distribution $\mu$; see \eqref{small_jumps_conv}. However, if we rule out this situation, we obtain the following result.
\begin{lemma} \label{measure_theoretic_lemma}
Let $\mu_n=(0,0,\lambda_n)$ be a sequence of infinitely divisible measures on $\Borel(H)$ converging weakly to $\mu=(0,0,\lambda)$. Then, it holds that
\[\lim_{n\rightarrow \infty}\int_H\Big(\norm{h}^2 \wedge 1\Big) \; \lambda_n\,({\rm d}h)=\int_H\Big(\norm{h}^2 \wedge 1\Big) \; \lambda\,({\rm d}h).\]
\end{lemma}

\begin{proof}
It is enough to prove the statement under the assumption that $\lambda_n$ is symmetric for all $n \in \mathbb{N}$, since the general case can be reduced back to this setting by a symmetrisation argument. Let $\delta \in C(\lambda)$ be such that $\delta \in (0,1)$. Then, it follows directly from Condition (\ref{conv_of_levy_measure}) that
\begin{align*} \label{outside_of_ball}
   \lim_{n\rightarrow \infty}\int_{\norm{h}>\delta}\Big(\norm{h}^2 \wedge 1 \Big)\; \lambda_n\,({\rm d}h)=\int_{\norm{h}>\delta}\Big(\norm{h}^2 \wedge 1 \Big)\; \lambda\,({\rm d}h).
\end{align*}
To establish the result for the integrals over the closed ball of radius $\delta$, note that the infinitely divisible measures $\tilde{\mu}_n=(0,0,\lambda_n\vert_{\Bar{B}_H(\delta)})$, where $\lambda_n\vert_{\Bar{B}_H(\delta)}=\lambda_n(\Bar{B}_H(\delta) \cap \,\cdot\,)$, converge weakly to the infinitely divisible measure $\tilde{\mu}=(0,0,\lambda\vert_{\Bar{B}_H(\delta)})$. Therefore, the set of measures $\{\tilde{\mu}_n\}_{n \in \mathbb{N}}$ is weakly compact and it follows from \cite[Le.\ VI.5.3]{parthasarathy_1967} that
\[\sup_{n \in \mathbb{N}}\int_H \norm{h}^4 \, \tilde{\mu}_n({\rm d}h)< \infty.\]
By the Skorokhod representation theorem and the Vallée-Poussin theorem \cite[Le.\ Th. II.22]{dellacherie_meyer_1978}, it follows as $n\to\infty$ that 
\begin{equation*}
  \int_{\Bar{B}_H(\delta)} \norm{h}^2 \, \lambda_n({\rm d}h)=\int_H \norm{h}^2 \, \tilde{\mu}_n({\rm d}h)\to \int_H \norm{h}^2 \, \tilde{\mu}({\rm d}h)=\int_{\Bar{B}_H(\delta)} \norm{h}^2 \, \lambda({\rm d}h),
\end{equation*}
which completes the proof.
\end{proof}

\section{Deterministic integrands}


The definition of the stochastic integral  for deterministic integrands with respect to a canonical $\alpha$-stable cylindrical L\'evy process $L$ depends on two classes of step functions. We give in the following a precise definition of what is meant by a step function.
\begin{definition} \hfill 
\begin{enumerate}
  \item[{\rm (1)}] An $L_2(G,H)$-valued step function is of the form
    \begin{align} \label{eq.det-step-HS}
		\psi \colon [0,T]\rightarrow L_2(G,H),\qquad \psi(t)=F_0\mathbb{1}_{\{0\}}(t)+\sum_{i=1}^{n-1} F_i \mathbb{1}_{(t_i,t_{i+1}]}(t),
	\end{align}
  where $0=t_1<\cdots < t_n=T$, $F_i \in L_2(G,H)$ for each $i \in \{0,...,n-1\}$. The space of $L_2(G,H)$-valued step functions is denoted by $\mathcal{S}_{\rm det}^{\rm HS}:=\mathcal{S}_{\rm det}^{\rm HS}(G,H)$.
  \item[{\rm (2)}] An $L(H,H)$-valued step function is of the form
  \begin{align}
		\gamma \colon [0,T]\rightarrow L(H,H),\qquad \gamma(t)=F_0\mathbb{1}_{\{0\}}(t)+\sum_{i=1}^{n-1} F_i \mathbb{1}_{(t_i,t_{i+1}]}(t),
	\end{align}
   where $0=t_1<\cdots < t_n=T$ and $F_i \in L(H,H)$ for each $i \in \{0,...,n-1\}$. The space of $L(H,H)$-valued step functions with $\sup_{t\in [0,T]}\norm{\gamma(t)}_{H \rightarrow H}\leq 1$ is denoted by $\mathcal{S}^{\rm 1,op}_{\rm det}:=\mathcal{S}^{\rm 1,op}_{\rm det}(H,H)$.
\end{enumerate}
\end{definition}
Consider an increment $L(t_{i+1})-L(t_i)$ of the cylindrical Lévy process $L$ and let $F_i \in L_2(G,H)$ for each $i \in \{1,...,n-1\}$. Since Hilbert-Schmidt operators are $0$-Radonifying by \cite[Th.\ VI.5.2]{vakhania_1981}, it follows from  \cite[Pr.\ VI.5.3]{vakhania_1981} that there exist a genuine random variables $\Phi_i\big(L(t_{i+1})-L(t_i)\big)\colon \Omega\to H$  for each $i \in \{1,...,n-1\}$  satisfying
\[(L(t_{i+1})-L(t_i))(\Phi_i^*h)=\langle \Phi_i(L(t_{i+1})-L(t_i)),h \rangle \quad \text{$P$-a.s.\ for all } h \in H.\]
We call the random variables  $\Phi_i\big(L(t_{i+1})-L(t_i)\big)$ for each $i \in \{1,...,n-1\}$ Radonified increments. 
The stochastic integral is defined for any $\psi \in \mathcal{S}_{\rm det}^{\rm HS}$ with representation (\ref{eq.det-step-HS}) as the sum of the Radonified increments 
\[I(\psi):=\int_0^T \psi \, \d L =\sum_{i=1}^{n-1} \Phi_i (L(t_{i+1})-L(t_i)).\]
Thus, the integral $I(\psi):\Omega\rightarrow H$ is a genuine $H$-valued random variable.

The following definition of the stochastic integral originates from the theory of vector measures, and was adapted to the probabilistic setting in \cite{urbanik_woyczynski_1967} by Urbanik and Woyczy{\'n}ski.

\begin{definition} \label{det_integrability}
A function $\psi\colon [0,T]\rightarrow L_2(G,H)$ is integrable if there exists a sequence $(\psi_n)_{n \in \mathbb{N}}$ of elements of $\mathcal{S}_{\rm det}^{\rm HS}$ satisfying
\begin{enumerate}[\rm(1)]
    \item $(\psi_n)_{n \in \mathbb{N}}$ converges to $\psi$ Lebesgue a.e.;  \label{det_int_def_1}
    \item $\displaystyle \lim_{m,n \rightarrow \infty}\sup_{\gamma \in \mathcal{S}^{\rm 1,op}_{\rm det}}E\Bigg[\norm{\int_0^T \gamma(\psi_m-\psi_n) \, \d L}\wedge 1 \Bigg]=0.$ \label{det_int_def_2}
\end{enumerate}
In this case,  the stochastic integral of the deterministic function $\psi$ is defined by
\[I(\psi):=\int_0^T \psi \, \d L:= \lim_{n\rightarrow \infty} \int_0^T \psi_n \, \d L \quad \text{in}\; L_P^0(\Omega,H).\]
The class of all deterministic integrable Hilbert-Schmidt operator-valued functions is denoted by $\mathcal{I}_{\rm det}^{\rm HS}:=\mathcal{I}_{\rm det}^{\rm HS}(G,H)$. 
\end{definition}

\begin{remark}\label{cadlag_remark}
If Conditions \eqref{det_int_def_1} and \eqref{det_int_def_2} in Definition \ref{det_integrability} are satisfied, then completeness of $L_P^0(\Omega,H)$ implies the existence of the limit. Furthermore, it follows that the integral process $(\int_0^t \psi \, \d L)_{t \geq 0}$, defined by $\int_0^t \psi\, \d L:=\int_0^T\1_{[0,t]}\psi\, \d L$ has c\'adl\'ag paths. To see this, note that for each $m, n \in \mathbb{N}$ the process $(\int_0^t \left(\psi_m - \psi_n \right)\, \d L)_{t \geq 0}$ has c\'adl\'ag paths. By an extension of \cite[Pr.\ 8.2.1]{kwapien_woyczynski_1992} to $H$-valued processes and Condition (\ref{det_int_def_2}) above, we obtain
\begin{multline*}
    \lim_{m,n \rightarrow \infty}P\Bigg(\sup_{0\leq t\leq T} \norm{\int_0^t \left(\psi_m - \psi_n\right)\, \d L} >\epsilon\Bigg)\\\leq 3 \lim_{m,n \rightarrow \infty} \sup_{0\leq t\leq T} P\Bigg( \norm{\int_0^t \left(\psi_m - \psi_n\right) \, \d L} >\frac{\epsilon}{3}\Bigg)=0.
\end{multline*}
By passing on to a suitable subsequence if necessary, we obtain that there exists a subsequence $(\int_0^t \psi_{n_k} \, \d L)_{k \in \mathbb{N}}$ that converges uniformly almost surely, which guarantees that the limiting process has c\'adl\'ag paths.
\end{remark}


The following is the main result of this section, which characterises the space $\mathcal{I}_{\rm det}^{\rm HS}$ of deterministic integrands, 
i.e.\ the set of functions $\psi \colon [0,T]\rightarrow L_2(G,H)$ that satisfy Definition \ref{det_integrability}
for the canonical $\alpha$-stable cylindrical L\'evy process $L$.

\begin{theorem} \label{det_if_and_only_if_integrable}
The space $\mathcal{I}_{\rm det}^{\rm HS}$ of deterministic functions  integrable with respect to the canonical $\alpha$-stable cylindrical L\'evy process in $G$ for $\alpha \in (0,2)$  coincides with $L_{\rm Leb}^{\alpha}\big([0,T],L_2(G,H)\big)$. 
\end{theorem}
The rest of this chapter is devoted to proving the above theorem which is divided into a few lemmas in the following. 
\begin{lemma} \label{L_alpha_smaller_than_modular}
Let $L$ be the canonical $\alpha$-stable cylindrical L\'evy process  in $G$ with cylindrical L\'evy measure $\lambda$.  Then there exists a constant $c_{\alpha}>0$ such that
\[\int_0^T \norm{\psi(t)}_{\rm HS}^{\alpha}\, {\rm d}t\leq c_{\alpha}\int_0^T\int_H \big(\norm{h}^2\wedge 1\big)\, (\lambda \circ \psi(t)^{-1}) ({\rm d}h) \, {\rm d} t\]
for all measurable functions $\psi:[0,T]\rightarrow L_2(G,H)$.
\end{lemma}

\begin{proof} Let $\Phi$ be an operator in $L_2(G,H)$.  
The spectral  theorem for compact operators, see e.g.\ \cite[Th.\ 4.1]{diestel_jarchow_tonge_1995},
guarantees that $\Phi$ has a decomposition of the form
\begin{align}\label{eq.spectralrep}
\Phi=\sum_{j=1}^{\infty}\gamma_j\langle a_j,\cdot \rangle b_j,
\end{align}
where $(a_j)_{j \in \mathbb{N}}$ and $(b_j)_{j \in \mathbb{N}}$ are orthonormal bases of $G$ and $H$, respectively,  and $(\gamma_j)_{j \in \mathbb{N}}$ is a sequence in $\mathbb{R}$. Let $\pi_n\colon H \rightarrow H$ be the projection onto $\text{Span}\{b_1,...,b_n\}$. 
We conclude from the spectral representation \eqref{spectral_represenation_of_stable_levy} of the stable measure $\lambda\circ \pi_{a_1,\dots, a_n}^{-1}$  for each $n\in {\mathbb N}$ that
\begin{align*}
    (\lambda \circ \Phi^{-1}\circ \pi^{-1}_n)(\Bar{B}_H^c)
    &=\lambda \Big(\Big\{g \in G:\sum_{j=1}^n\gamma_j^2 \langle a_j, g \rangle^2>1\Big\}\Big)
    \\
    &=\lambda \circ\pi_{a_1,...,a_n}^{-1}\Big(\Big\{x \in \mathbb{R}^n:\sum_{j=1}^n\gamma_j^2x_j^2>1\Big\}\Big)\\
    &=\frac{\alpha}{c_{\alpha}}\int_{S_{\mathbb{R}^n}}\int_0^{\infty}\mathbb{1}_{\{y \in \mathbb{R}^n:\sum_{j=1}^n\gamma_j^2y_j^2>1\}}(rx)\frac{1}{r^{1+\alpha}}\;\d r\; \nu_n(\d x)\\
    &=\frac{1}{c_{\alpha}}\int_{S_{\mathbb{R}^n}} \Big(\sum_{j=1}^n\gamma_j^2x_j^2\Big)^{\alpha/2}\nu_n\,({\rm d} x),
\end{align*}
where $\nu_n$ is a uniform distribution on the sphere $S_{{\mathbb R}^n}$ not necessarily of unit mass. 
By defining $c_n:=\sum_{j=1}^n\gamma_j^2$ and applying Jensen's inequality to the concave function $\beta \mapsto \beta^{\alpha/2}$ with respect to the discrete probability measure $\{c_n^{-1}\gamma_1^2,...,c_n^{-1}\gamma_n^2\}$, we  obtain
\begin{align*}
    \frac{1}{c_{\alpha}}\int_{S_{\mathbb{R}^n}} \Big(\sum_{j=1}^n\gamma_j^2x_j^2\Big)^{\alpha/2} \,\nu_n( {\rm d} x) &\geq \frac{c_n^{\alpha/2}}{c_{\alpha}}\int_{S_{\mathbb{R}^n}}\sum_{j=1}^n\frac{\gamma_j^2}{c_n}\vert x_j \vert^{\alpha}\,\nu_n( {\rm d} x). 
\end{align*}
Letting $\nu_n^1=\frac{1}{\nu_n(S_{\mathbb{R}^n})}\nu_n$, Lemma 2.4 and A2 in \cite{riedle_2018} imply 
\begin{align*}
\frac{c_n^{\alpha/2}}{c_{\alpha}}\int_{S_{\mathbb{R}^n}}\sum_{j=1}^n\frac{\gamma_j^2}{c_n}\vert x_j \vert^{\alpha}\;\nu_n({\rm d} x)
&=    \frac{c_n^{\alpha/2}}{c_{\alpha}} \nu_n(S_{\mathbb{R}^n}) \sum_{j=1}^n\frac{\gamma_j^2}{c_n}\int_{S_{\mathbb{R}^n}}\vert x_j\vert ^{\alpha}\nu_n^1({\rm d} x)\\
&=\frac{c_n^{\alpha/2}}{c_{\alpha}} \nu_n(S_{\mathbb{R}^n}) \frac{\Gamma(\frac{n}{2})\Gamma(\frac{1+\alpha}{2})}{\Gamma(\frac{1}{2})\Gamma(\frac{n+\alpha}{2})}\sum_{j=1}^n\frac{\gamma_j^2}{c_n}\\
&  =\frac{1}{c_{\alpha}} \Big( \sum_{j=1}^n\gamma_j^2\Big)^{\alpha/2}\\
&  = \frac{1}{c_{\alpha}} \Big(\sum_{j=1}^n\norm{\Phi a_j}^2\Big)^{\alpha/2},
\end{align*}
where the last step follows from the spectral representation \eqref{eq.spectralrep}. Thus, for all $n \in \mathbb{N}$ it holds that
\begin{align}\label{eq.estimate-finite-projection}
\frac{1}{c_{\alpha}} \Big(\sum_{j=1}^n\norm{\Phi a_j}^2\Big)^{\alpha/2}\leq(\lambda \circ \Phi^{-1}\circ \pi_n^{-1})(\Bar{B}_H^c).
\end{align}
Since $\pi_n \Phi \rightarrow \Phi$ in $L_2(G,H)$, Lemma \ref{le.conv_of_radonif_rv} implies that $\big((\pi_n\circ\Phi)(L(1))\big)_{n \in \mathbb{N}}$ converges in probability to the random variable $\Phi(L(1))$. Condition \eqref{conv_of_levy_measure} yields
\[\lim_{n \rightarrow \infty}(\lambda \circ \Phi^{-1}\circ \pi^{-1}_n)(\Bar{B}_H^c)=(\lambda \circ \Phi^{-1})(\Bar{B}_H^c).\]
By taking limits as $n \rightarrow \infty$ on both sides in \eqref{eq.estimate-finite-projection}, we obtain
\[\frac{1}{c_{\alpha}} \norm{\Phi}_{\rm HS}^{\alpha}\leq (\lambda \circ \Phi^{-1})(\Bar{B}_H^c).\]
It follows for any measurable function $\psi:[0,T]\rightarrow L_2(G,H)$ that
\begin{align*}
    \int_0^T \norm{\psi(t)}_{\rm HS}^{\alpha}\,{\rm d}t
    &\leq c_{\alpha}\int_0^T\int_{\Bar{B}_H^c}(\lambda \circ \psi(t)^{-1})\,({\rm d}h)\,{\rm d}t \nonumber\\
    &\leq c_{\alpha}\int_0^T\int_H \big(\norm{h}^2\wedge 1\big)(\lambda \circ \psi(t)^{-1})\,({\rm d}h)\,{\rm d}t, 
\end{align*}
which completes the proof.
\end{proof}
The product measure of two cylindrical measures is defined analogously to the case of Radon measures; see \cite[Ch.\ II.2.2]{schwartz_1973}. The following lemma provides an alternative representation of an integral with 
respect to the product measure of the cylindrical L\'evy measure of $L$ and the Lebesgue measure on a finite interval. 
To make sense out of this, the Lebesgue measure is considered as a cylindrical measure on $\Borel({\mathbb{R}})$. 
\begin{lemma}\label{alternative_form_of_modular}
Let $L$ be the canonical $\alpha$-stable cylindrical L\'evy process  in $G$ with cylindrical L\'evy measure $\lambda$. 
Then we have for each $\psi \in \mathcal{S}_{\rm det}^{\rm HS}$ with $\psi(0)=0$ that
\[\int_0^T \int_H \Big( \norm{h}^2 \wedge 1\Big) \big( \lambda \circ \psi(t)^{-1} \big)\, ({\rm d}h)\, {\rm d}t= \int_H \Big( \norm{h}^2 \wedge 1\Big)\Big(\big(\lambda \otimes {\rm Leb} \big)\circ \kappa_\psi^{-1} \Big)\, ({\rm d}h),\]
where $\kappa_\psi\colon  G \times [0,T]\rightarrow H$ is defined by $\kappa_\psi(g,t)=\psi(t)g$.
\end{lemma}
\begin{proof}
First, we show that the result holds for $\psi=\Phi\mathbb{1}_{(t_i,t_{i+1}]}$, where $\Phi \in L_2(G,H)$ and $0 \leq t_i<t_{i+1}\leq T$. In this case, we see that for all $C \in \mathcal{Z}_*(H)$
\begin{align} \label{EQNM}
 \big(\lambda \!\otimes\! {\rm Leb} \big)\circ \kappa_\psi^{-1} (C)&=\big(\lambda \otimes {\rm Leb} \big)\big( \Phi^{-1}(C) \times (t_i,t_{i+1}] \big)
  =(t_i-t_{i+1})\big(\lambda \circ \Phi^{-1}\big)(C).
\end{align}
Since the cylindrical measure on the right hand side of Equation (\ref{EQNM}) is the cylindrical Lévy measure of the Radonified increment $\Phi(L(t_{i+1})-L(t_i))$, it extends to a genuine Lévy measure on $\Borel(H)$ for which we keep the notation $\lambda \circ \Phi^{-1}$. Consequently, the cylindrical Lévy measure on the left hand side of Equation (\ref{EQNM}) extends to a genuine Lévy measure on $\Borel(H)$, and the two extensions agree on $\Borel(H)$. It follows that
\begin{align} \label{int_equivalency}
   & \int_H \Big( \norm{h}^2 \wedge 1\Big)\big(\big(\lambda \otimes {\rm Leb} \big)\circ \kappa_\psi^{-1} \big)\, ({\rm d}h)\\
    & =\int_{t_i}^{t_{i+1}} \int_H \Big( \norm{h}^2 \wedge 1\Big)\big(\lambda \circ \Phi^{-1}\big)\, ({\rm d}h) \, {\rm d}t 
    =\int_0^T \int_H \Big( \norm{h}^2 \wedge 1\Big)\big(\lambda \circ \psi(t)^{-1}\big)\, ({\rm d}h) \, {\rm d}t. \nonumber
\end{align}
Let $\psi \in \mathcal{S}_{\rm det}^{\rm HS}$ be of the form as in \eqref{eq.det-step-HS} with $\psi(0)=0$. 
For each $C \in \mathcal{Z}_*(H)$ we obtain
\begin{align*}
    \kappa_\psi^{-1}(C)&= \bigcup_{i=1}^{n-1} \Big\{(g,t)\in G \times [0,T]:\Phi_ig \mathbb{1}_{(t_i,t_{i+1}] }\in C\Big\}.
\end{align*}
Since the above is a finite union of disjoint cylindrical sets, it follows
\begin{align} \label{EQNO}
    \Big((\lambda \otimes {\rm Leb}) \circ \kappa_\psi^{-1}\Big)(C)&=\sum_{i=1}^{n-1}\Big((\lambda \otimes {\rm Leb})\circ \kappa_{\Phi_i \mathbb{1}_{(t_i,t_{i+1}]}}^{-1}\Big)(C).
\end{align}
As the measure on the right side of Equation (\ref{EQNO}) extends to a genuine Lévy measure on $\Borel(H)$ according to the first part of this proof, the measure on the left extends to a genuine Lévy measure on $\Borel(H)$. It follows from Equation (\ref{int_equivalency}) that
\begin{align*}
    \int_H \Big( \norm{h}^2 \wedge 1\Big)&\Big(\big(\lambda \otimes {\rm Leb} \big)\circ \kappa_\psi^{-1} \Big)\, ({\rm d}h)\\
    &= \sum_{i=1}^{n-1}\int_H \Big( \norm{h}^2 \wedge 1\Big)\Big(\big(\lambda \otimes {\rm Leb} \big)\circ \kappa_{\Phi_i \mathbb{1}_{(t_i,t_{i+1}]}}^{-1} \Big)\, ({\rm d}h)\\
    &=\sum_{i=1}^{n-1} \int_0^T\int_H\Big( \norm{h}^2 \wedge 1\Big)\big(\lambda \circ (\Phi_i \mathbb{1}_{(t_i,t_{i+1}]}(t))^{-1}\big)\, ({\rm d}h)\, {\rm d}t\\
    &=\int_0^T \int_H \Big( \norm{h}^2 \wedge 1\Big) \big( \lambda \circ \psi(t)^{-1} \big)\, ({\rm d}h)\, {\rm d}t, 
\end{align*}
which completes the proof.
\end{proof}

\begin{lemma} \label{le.det-small-if-small}
	Let $L$ be the canonical $\alpha$-stable cylindrical L\'evy process  in $G$
and $(\psi_n)_{n\in{\mathbb N}}$ a sequence in $\mathcal{S}_{\rm det}^{\rm HS}$. Then the following are 
equivalent:
\begin{enumerate}
	\item[{\rm (a)}] $\displaystyle \lim_{n\to\infty}\norm{\psi_n}_{L^{\alpha}}=0$;
	\item[{\rm (b)}] $\displaystyle \lim_{n\to\infty} \sup_{\gamma \in \mathcal{S}^{\rm 1,op}_{\rm det}}E\Bigg[\norm{\int_0^T \gamma \psi_n \, dL}\wedge 1 \Bigg]=0$.
\end{enumerate}
\end{lemma}

\begin{proof}
 Let $(\psi_n)_{n \in \mathbb{N}}$ be a sequence of elements of $\mathcal{S}_{\rm det}^{\rm HS}$ with $\lim_{n \rightarrow \infty}\norm{\psi_n}_{L^\alpha}=0$.
 Corollary 3 in \cite{kosmala2021stochastic} and Markov's inequality show for all $\epsilon>0$ and $p<\alpha$ that
\begin{align*}
    \sup_{\gamma \in \mathcal{S}^{\rm 1,op}_{\rm det}}P\Bigg(\norm{\int_0^T\gamma \psi_n \; \d L}>\epsilon\Bigg) &\leq \frac{1}{\epsilon^p}\sup_{\gamma \in \mathcal{S}^{\rm 1,op}_{\rm det}}E\Bigg[\norm{\int_0^T\gamma \psi_n\; \d L}^p\Bigg]\\
    &\leq \sup_{\gamma \in \mathcal{S}^{\rm 1,op}_{\rm det}} \frac{1}{\epsilon^p} c_{\alpha,p}\Bigg(\int_0^T \norm{\gamma \psi_n}_{\rm HS}^{\alpha}\, {\rm d}t\Bigg)^{p/\alpha}
    \le \frac{1}{\epsilon^p}c_{\alpha,p}\norm{\psi_n}_{L^{\alpha}}^{p},
\end{align*}
where the last inequality follows from $\norm{\gamma\psi_n}_{\rm HS}\le \norm{\psi_n}_{\rm HS}$ for $\gamma\in \mathcal{S}^{\rm 1,op}_{\rm det}$. This proves (b) because of the equivalent characterisation of the topology in $L^0_P(\Omega,H)$.

To prove the converse implication, let $(\psi_n)_{n \in \mathbb{N}}$ be a sequence of elements of $\mathcal{S}_{\rm det}^{\rm HS}$ satisfying Condition (b).  Since for each $n \in \mathbb{N}$, $\psi_n$ has a representation of the form
\[\psi_n(t)=\Phi_0^n\mathbb{1}_{\{0\}}(t)+\sum_{i=1}^{N(n)-1} \Phi_i^n \mathbb{1}_{(t_i^n,t_{i+1}^n]}(t),\]
\noindent where $0= t_1^n<...<t_{N(n)}^n = T$, and $\Phi_i^n \in L_2(G,H)$ for each $i \in \{0,...,N(n)-1\}$, the integral $I(\psi_n)$ satisfies
\[I(\psi_n)=\sum_{i=1}^{N(n)-1}\Phi_i^n \big(L(t_{i+1}^n)-L(t_i^n)\big). \] 
The cylindrical L\'evy measure of $\Phi_i^n (L(t_{i+1}^n)-L(t_i^n))$ is given by 
$ (t_{i+1}^n-t_i^n) \big(\lambda \circ (\Phi_i^n)^{-1}\big)$ which extends to a genuine L\'evy measure on 
$\Borel(H)$ for which we keep the same notation. Independent increments of $L$ together with Equation \eqref{EQNO} show that the infinitely divisible random variable $I(\psi_n)$ has characteristics
\[\Bigg(0,0,\sum_{i=1}^{N(n)-1} (t_{i+1}^n-t_i^n) \big(\lambda \circ (\Phi_i^n)^{-1}\big)\Bigg)=\Bigg(0,0,(\lambda \otimes {\rm Leb}) \circ \kappa_{\psi_n}^{-1}\Bigg).\]
Since Condition (b) implies that $\lim_{n \rightarrow \infty}I(\psi_n)=0$ in $L_P^0(\Omega,H)$ and $\big(\lambda \otimes {\rm Leb} \big)\circ \kappa_{\psi_n}^{-1}$ is the L\'evy measure of $I(\psi_n)$, we conclude from Lemmata \ref{L_alpha_smaller_than_modular}, \ref{alternative_form_of_modular} and  \ref{measure_theoretic_lemma} that
\begin{align*}
    \lim_{n \rightarrow \infty}\int_0^T \norm{\psi_n(t)}_{\rm HS}^{\alpha}\, {\rm d}t&\leq \lim_{n \rightarrow \infty} c_{\alpha}\int_0^T\int_H \big(\norm{h}^2\wedge 1\big)(\lambda \circ \psi_n(t)^{-1})\, ({\rm d}h)\, {\rm d}t\\
    &=\lim_{n \rightarrow \infty}c_{\alpha}\int_H \Big( \norm{h}^2 \wedge 1\Big)\Big(\big(\lambda \otimes {\rm Leb} \big)\circ \kappa_{\psi_n}^{-1} \Big)\, ({\rm d}h)=0, 
\end{align*}
which completes the proof.
\end{proof}

\begin{proof}[Proof of Theorem \ref{det_if_and_only_if_integrable}]
 If $\psi \in \mathcal{I}_{\rm det}^{\rm HS}$ then there exists a sequence $(\psi_n)_{n \in \mathbb{N}} \subseteqq\mathcal{S}_{\rm det}^{\rm HS}$ such that $\psi_n \rightarrow \psi$ Lebesgue a.e.\ and  $\sup_{\gamma \in \mathcal{S}_{\rm det}^{1,op}}E[\norm{I(\gamma(\psi_m-\psi_n))}\wedge 1] \rightarrow 0$ as $m, n \to \infty$. This implies $\norm{\psi_m-\psi_n}_{L^{\alpha}} \rightarrow 0$  by Lemma \ref{le.det-small-if-small}. Completeness of $L^\alpha$ and the fact that $\psi_n \rightarrow \psi$ Lebesgue a.e.\ allows us to conclude $\psi \in L^{\alpha}$. Conversely, if  $\psi \in L^{\alpha}$ then there exists a sequence $(\psi_n)_{n \in \mathbb{N}}$ of elements in $\mathcal{S}_{\rm det}^{\rm HS}$ such that $\psi_n \rightarrow \psi$ Lebesgue a.e.\ and $\norm{\psi_n-\psi}_{L^{\alpha}}\rightarrow 0$ by an extension of  \cite[Re.\ 1.2.20]{hytonen_neerven_2016} to all $\alpha \in (0,1)$. It follows that $\norm{\psi_m-\psi_n}_{L^{\alpha}}\rightarrow 0$, which implies  $\sup_{\gamma \in \mathcal{S}_{\rm det}^{\rm 1,op}}E[\norm{I(\gamma(\psi_m-\psi_n))}\wedge 1] \rightarrow 0$ by Lemma \ref{le.det-small-if-small} and establishes  $\psi \in \mathcal{I}_{\rm det}^{\rm HS}$. 
\end{proof}

\section{Predictable Integrands}


As in the case of deterministic integrands, we begin by introducing two classes of functions on which our definition of the stochastic integral depend.
\begin{definition} \hfill 
	\begin{enumerate}
		\item[{\rm (1)}] An $L_2(G,H)$-valued adapted step process $\Psi \colon \Omega \times [0,T]\rightarrow L_2(G,H)$ is of the form 
		\begin{align} \label{eq.step-HS}
		  \Psi(\omega, t)=\left(\sum_{k=1}^{N(0)}F_{0,k}\mathbb{1}_{A_{0,k}}(\omega)\right)\mathbb{1}_{\{0\}}(t)+\sum_{i=1}^{n-1} \left(\sum_{k=1}^{N(i)}F_{i,k}\mathbb{1}_{A_{i,k}}(\omega)\right) \mathbb{1}_{ (t_i,t_{i+1}]}(t),
		\end{align}
		where $0=t_1<\cdots < t_{n}=T$, $A_{0,k} \in \mathcal{F}_0$ and $F_{0,k} \in L_2(G,H)$ for all $k=1,...,N(0)$, $A_{i,k} \in \mathcal{F}_{t_i}$ and $F_{i,k} \in L_2(G,H)$ for all $i=1,...,n-1$ and $k=1,...,N(i)$. The space of all $L_2(G,H)$-valued adapted step processes is denoted by $\mathcal{S}_{\rm adp}^{\rm HS}:=\mathcal{S}_{\rm adp}^{\rm HS}(G,H)$.
		\item[{\rm (2)}] An $L(H,H)$-valued adapted step process $\Gamma \colon \Omega \times [0,T]\rightarrow L(H,H)$ is of the form
		\begin{align}
		  \Gamma (\omega, t)=\left(\sum_{k=1}^{N(0)}F_{0,k}\mathbb{1}_{A_{0,k}}(\omega)\right)\mathbb{1}_{\{0\}}(t)+\sum_{i=1}^{n-1} \left(\sum_{k=1}^{N(i)}F_{i,k}\mathbb{1}_{A_{i,k}}(\omega)\right) \mathbb{1}_{ (t_i,t_{i+1}]}(t),
		\end{align}
			where $0=t_1<\cdots < t_{n}=T$, $A_{0,k} \in \mathcal{F}_0$ and $F_{0,k}\in L(H,H)$ for all $k=1,...,N(0)$, $A_{i,k} \in \mathcal{F}_{t_i}$ and $F_{i,k} \in L(H,H)$ for all $i=1,...,n-1$ and $k=1,...,N(i)$. The space of all $L(H,H)$-valued adapted step processes with
			\begin{align*}
			    \sup_{(\omega,t)\in \Omega \times [0,T]}\norm{\Gamma(\omega,t)}_{H \rightarrow H}\leq 1
			\end{align*}
			is denoted by  $\mathcal{S}_{{\rm adp}}^{1, {\rm op}}:=\mathcal{S}_{{\rm adp}}^{1, {\rm op}}(H,H)$.
	\end{enumerate}
\end{definition}

Let $\Psi \in \mathcal{S}_{\rm adp}^{\rm HS}$ be of the form (\ref{eq.step-HS}). Since Hilbert-Schmidt operators are $0$-Radonifying by \cite[Th.\ VI.5.2]{vakhania_1981}, it follows from  \cite[Pr.\ VI.5.3]{vakhania_1981} that there exists an $H$-valued random variable $\Phi_{i,k}(L(t_{i+1})-L(t_i))\colon \Omega\to H$ for each $i=1,...,n-1$ and $k=1,...,N(i)$, satisfying \[ \big(L(t_{i+1})-L(t_i)\big)(\Phi_{i,k}^*h)=\langle \Phi_{i,k}(L(t_{i+1})-L(t_i)),h \rangle \quad \text{$P$-a.s.\ for all}\; h \in H.\]
In this case,  the stochastic integral of $\Psi$ is defined by
\[I(\Psi):=\int_0^T \Psi(t) \, {\rm d}L(t) :=\sum_{i=1}^{n-1}\sum_{k=1}^{N(i)} \mathbb{1}_{A_{i,k}}\Phi_{i,k} (L(t_{i+1})-L(t_i) ).\]
Thus, the integral $I(\Psi):\Omega\rightarrow H$ is a genuine $H$-valued random variable.

For the purposes of this section, it is convenient to introduce the measure space $\big(\Omega \times [0,T],\mathcal{P}, P_T\big)$, where $\mathcal{P}$ denotes the predictable $\sigma$-algebra and the measure $P_T$ is defined by $P_T:=P \otimes {\rm Leb} \vert_{[0,T]}$.

\begin{definition} \label{pred_integrability}
We say that a predictable process $\Psi$ is integrable if there exists a sequence $(\Psi_n)_{n \in \mathbb{N}}$ of processes in $\mathcal{S}_{\rm adp}^{\rm HS}$ such that
\begin{enumerate}[(1)]
    \item[{\rm (1)}]  $(\Psi_n)_{n \in \mathbb{N}}$ converges $P_T$-a.e. to $\Psi$,
    \item[{\rm (2)}] $\displaystyle \lim_{m,n \rightarrow \infty}\sup_{\Gamma \in \mathcal{S}_{{\rm adp}}^{1, {\rm op}}}E\Bigg[\norm{\int_0^T \Gamma(\Psi_m-\Psi_n) \;{\rm d}L}\wedge1 \Bigg]=0.$
\end{enumerate}
In this case,  the stochastic integral of $\Psi$ is defined by
\[I(\Psi):=\int_0^T \Psi \;{\rm d}L= \lim_{n\rightarrow \infty} \int_0^T \Psi_n \;{\rm d}L \quad \text{in}\;L_P^0(\Omega,H).\]
The class of all integrable $L_2(G,H)$-valued predictable processes is denoted by $\mathcal{I}_{\rm prd}^{\rm HS}:=\mathcal{I}_{\rm prd}^{\rm HS}(G,H)$. 
\end{definition}
The space $L_P^0\big(\Omega,L_{{\rm Leb}}^{\alpha}\big([0,T],L_2(G,H)\big)\big)$ for $\alpha\in (0,2)$ denotes the set of all random variables $\Psi:\Omega \rightarrow L_{{\rm Leb}}^{\alpha}\big([0,T],L_2(G,H)\big)$ endowed with the translation invariant metric $\normm{\cdot}_{L^{\alpha}}$ defined by
\[\normm{\Psi_1-\Psi_2}_{L^{\alpha}}=E\Big[\norm{\Psi_1-\Psi_2}_{L^{\alpha}}\wedge 1\Big].\]
The space $L_{{\rm Leb}}^{\alpha}\big([0,T],L_2(G,H)\big)$ is separable, see \cite[Pr.\ 1.2.29]{hytonen_neerven_2016} for the proof when $\alpha \geq 1$, which can be generalised for all $\alpha>0$. If $(\Psi(t))_{t \geq 0}$ is an $L_2(G,H)$-valued predictable stochastic process that almost surely has paths in $L_{{\rm Leb}}^{\alpha}\big([0,T],L_2(G,H)\big)$, then \cite[Pr.\ 3.19]{da_prato_zabczyk_1992} guarantees that the mapping $\Psi: \Omega \rightarrow L_{{\rm Leb}}^{\alpha}\big([0,T],L_2(G,H)\big)$ defines a random variable. Therefore, it is reasonable to write $\Psi \in L_P^0\big(\Omega,L_{{\rm Leb}}^{\alpha}\big([0,T],L_2(G,H)\big)\big)$ to denote this setting.

\begin{lemma} \label{pred_density}
Let $\Psi$ be a predictable stochastic process in $L_P^0\big(\Omega,L_{{\rm Leb}}^{\alpha}\big([0,T],L_2(G,H)\big)\big)$. Then there exists a sequence $(\Psi_k)_{k \in \mathbb{N}}$ of elements of $\mathcal{S}_{\rm adp}^{\rm HS}$
converging to $\Psi$ both in $\normm{\cdot}_{L^{\alpha}}$ and $P_T$-a.e. 
\end{lemma}

\begin{proof}
If $\Psi$ is bounded, then $\Psi \in L_{P_T}^{\infty}\big(\Omega \times [0,T],L_2(G,H)\big)$. Since the algebra of sets
\[\mathcal{A'}=\big\{(s,t]\times B:s<t,B \in \mathcal{F}_s\big\} \cup \big\{\{0\}\times B:B \in \mathcal{F}_0\big\}\]
generates $\mathcal{P}$, we conclude from \cite[Le.\ 1.2.19]{hytonen_neerven_2016} and \cite[Re.\ 1.2.20]{hytonen_neerven_2016} that there exists a sequence $(\Psi_k)_{k \in \mathbb{N}}$ of uniformly bounded processes in $\mathcal{S}_{\rm adp}^{\rm HS}$ such that $\Psi_k \rightarrow \Psi$ $P_T$-a.e. Thus, there exists a set $N \in \mathcal{P}$ such that $P_T(N)=0$ and $\big(\Psi_k(\omega,t)-\Psi(\omega,t)\big)\rightarrow 0$ for all $(\omega,t) \in N^c$. Fubini's theorem implies that
\[P_T(N)=P \otimes {\rm Leb}\vert_{[0,T]}(N)=\int_{\Omega}{\rm Leb}\vert_{[0,T]}(N_{\omega})\,P({\rm d}\omega)=0,\]
where for each fixed $\omega \in \Omega$ we define
\[N_\omega:=\Big\{t \in [0,T]\colon  \big(\Psi_k(\omega,t)-\Psi(\omega,t)\big)_{m \in \mathbb{N}} \; \text{does not converge to 0}\Big\}.\]
 The above implies that ${\rm Leb}\vert_{[0,T]}(N_\omega)=0$ for almost all $\omega \in \Omega$, that is, there exists an $\Omega_0\subseteq \Omega$ with $P(\Omega_0)=1$ such that for all $\omega \in \Omega_0$ we have 
 \[{\rm Leb}\vert_{[0,T]}\Big(t \in [0,T]\colon  \big(\Psi_k(\omega,t)-\Psi(\omega,t)\big)_{m \in \mathbb{N}} \; \text{does not converge to 0}\Big)=0.\]
Because  $(\Psi_k)_{k \in \mathbb{N}}$ is uniformly bounded and $\Psi$ is bounded, we can conclude from   Lebesgue's dominated convergence theorem  that $\norm{\Psi_k(\omega,\cdot)-\Psi(\omega,\cdot)}_{L^{\alpha}}\rightarrow 0$ as $k \rightarrow \infty$ for each $\omega \in \Omega_0$.  Another application of Lebesgue's dominated convergence theorem yields
\begin{align*}
    \lim_{k \rightarrow \infty}\normm{\Psi-\Psi_k}_{L^{\alpha}}= \lim_{k \rightarrow \infty} \int_{\Omega}\Big(\norm{\Psi-\Psi_k}_{L^{\alpha}}\wedge 1\Big)\,{\rm d}P=0,
\end{align*}
which shows the claim  if $\Psi$ is bounded. In the case of a general $\Psi$, we define
\[\Psi_n:\Omega\times [0,T] \rightarrow L_2(G,H),\qquad
    \Psi_n(\omega,t)= 
\begin{cases}
    \Psi(\omega,t) & \text{if } \norm{\Psi(\omega,t)}_{\rm HS}\leq n , \\
    0              & \text{otherwise.}
\end{cases}
\]
Clearly, $\lim_{n \rightarrow \infty}\normm{\Psi-\Psi_n}_{L^{\alpha}}=0$. The first part of the proof shows that 
for each $n \in \mathbb{N}$ there exists a sequence $(\Psi_{n,k})_{k \in \mathbb{N}} \subseteq \mathcal{S}_{\rm adp}^{\rm HS}$ converging to $\Psi_n$ as $k\to\infty$  in $\normm{\cdot}_{L^{\alpha}}$ and $P_T$-a.e. For each $n \in \mathbb{N}$  choose $k_n\in {\mathbb N}$ such that $\normm{(\Psi_n-\Psi_{n,k_n})}_{L^{\alpha}} < \frac{1}{n}$.
It follows that 
\[\lim_{n \rightarrow \infty}\normm{(\Psi-\Psi_{n,k_n})}_{L^{\alpha}}\leq \lim_{n \rightarrow \infty}\bigg(\normm{(\Psi-\Psi_n)}_{L^{\alpha}}+\normm{(\Psi_n-\Psi_{n,k_n})}_{L^{\alpha}}\bigg)=0,\]
which completes the proof. 
\end{proof}

\section{Construction of the decoupled tangent sequence}

The technique of constructing decoupled tangent sequences is a powerful tool to obtain strong results on a sequence of possibly dependent random variables. In this section, we briefly recall the fundamental definition, see e.g.\ Kwapie{\'n} and Woyczy{\'n}ski \cite{kwapien_woyczynski_1992} or  de la Pe\~{n}a and Gin\'{e} \cite{Pena_1999}, and construct the decoupled tangent sequence in our setting which will enable us to identify the largest space of predictable integrands in the next section. 

\begin{remark} \label{product_expectation_remark}
We repeatedly use the fact in the following that given a random variable $X$ on $(\Omega,\mathcal{F},P)$ and another probability space $(\Omega',\mathcal{F'},P')$, the random variable $X$ can always be considered as a random variable on the product space $(\Omega \times \Omega',\mathcal{F} \otimes \mathcal{F}',P \otimes P')$ by defining
\[X(\omega,\omega')=X(\omega)\quad \text{for all}\;(\omega,\omega')\in \Omega \times \Omega'.\]
In this case, if $X$ is real-valued and $P$-integrable we have  $E_P[X]=E_{P \otimes P'}[X]$.
\end{remark}

\noindent In the next definition, we follow closely Chapter 4.3 of \cite{kwapien_woyczynski_1992}.

\begin{definition}\label{dec_tang_seq_def}
Let $\big(\Omega, \mathcal{F},P,(\mathcal{F}_n)_{n\in{\mathbb N}}\big)$ be a filtered probability space and $(X_n)_{n \in \mathbb{N}}$ an $(\mathcal{F}_n)$-adapted sequence of $H$-valued random variables. If $\big(\Omega', \mathcal{F'},P',(\mathcal{F'}_n)_{n\in{\mathbb N}}\big)$ is another filtered probability space, then a sequence $(Y_n)_{n \in \mathbb{N}}$ of $H$-valued random variables defined on $\big(\Omega \times \Omega^\prime, \mathcal{F}\otimes \mathcal{F'}, P \otimes P', (\mathcal{F}_n\otimes \mathcal{F'}_n)_{n\in{\mathbb N}}\big)$ is said to be a decoupled tangent sequence to $(X_n)_{n \in \mathbb{N}}$ if
\begin{enumerate}[(1)]
    \item[{\rm (1)}]  for each $\omega \in \Omega$, we have that $(Y_n(\omega,\cdot))_{n \in \mathbb{N}}$ is a sequence of independent random variables on $(\Omega',\mathcal{F}',P')$;
    \item[{\rm (2)}]  the sequences $(X_n)_{n \in \mathbb{N}}$ and $(Y_n)_{n \in \mathbb{N}}$ satisfy for each $n\in {\mathbb N}$ that
\[\mathcal{L}(X_n \vert \mathcal{F}_{n-1}\otimes \mathcal{F'}_{n-1})=\mathcal{L}(Y_n \vert \mathcal{F}_{n-1}\otimes \mathcal{F'}_{n-1}) \quad P \otimes P'-\text{a.s.}\]
\end{enumerate}
\end{definition}

The main tool for establishing the stochastic integral in the next section is a cylindrical L\'evy process $\widetilde{L}$ on an enlarged probability space, whose Radonified increments are decoupled to the Radonified increments of the original canonical $\alpha$-stable cylindrical L\'evy process. This cylindrical L\'evy process $\widetilde{L}$ is explicitly constructed in the following result.

\begin{proposition} \label{dec_tan_seq}
Let $L$ be a cylindrical L\'evy process on $G$, $0 = t_0 \leq ... \leq t_N=T$ be a partition of $[0,T]$	
and for each $n=1,...,N$ we define $\Theta_n:= \sum_{k=1}^{M(n)} \Phi_{n,k} \mathbb{1}_{A_{n,k}}$, where $\Phi_{n,k} \in L_2(G,H)$,  $A_{n,k} \in \mathcal{F}_{t_{n-1}}$ for all $k=1,...,M(n)$.
By defining cylindrical random variables 
\[ \widetilde{L}(t)\colon G \rightarrow L^0_{P \otimes P}(\Omega \times \Omega; \mathbb{R}), \qquad 
\Big(\widetilde{L}(t)g\Big)(\omega,\omega')=\Big(L(t)g\Big)(\omega'),\]
it follows that $(\widetilde{L}(t):t\geq 0)$  is a cylindrical L\'evy process on $G$  and the sequence of its Radonified increments
\[\left(\Theta_n \big(\widetilde{L}(t_n)-\widetilde{L}(t_{n-1})\big)\right)_{n \in \{1,...,N\} }\]
defined on $\big(\Omega \times \Omega,\mathcal{F} \otimes \mathcal{F},P \otimes P,(\mathcal{F}_{t_n} \otimes \mathcal{F}_{t_n})_{n \in \{0,...,N\}}\big)$ is a decoupled tangent sequence to the sequence of Radonified increments
\[\Big(\Theta_n \big(L(t_n)-L(t_{n-1})\big)\Big)_{n \in \{1,...,N\}}\]
defined on $\big(\Omega,\mathcal{F},P,(\mathcal{F}_{t_n})_{n \in \{0,...,N\}}\big).$
\end{proposition}

\begin{proof}
In order to make it easier to follow this proof, we define $\Omega'=\Omega$, $\mathcal{F'}=\mathcal{F}$, $P'=P$ and $\mathcal{F'}_{t_n}=\mathcal{F}_{t_n}$ for all $n \in \{0,...,N\}$ and instead of denoting the filtered product space by
\[\Big(\Omega \times \Omega,\mathcal{F} \otimes \mathcal{F},P \otimes P,(\mathcal{F}_{t_n} \otimes \mathcal{F}_{t_n})_{n \in \{0,...,N\}}\Big),\]
we write
\[\Big(\Omega \times \Omega',\mathcal{F} \otimes \mathcal{F'},P \otimes P',(\mathcal{F}_{t_n} \otimes \mathcal{F'}_{t_n})_{n \in \{0,...,N\}}\Big).\]
The fact that for each $t\geq 0$ the mapping $\widetilde{L}(t)\colon G \rightarrow L_{P \otimes P'}^0(\Omega \times \Omega',\mathbb{R})$ is continuous follows directly from the definition of $\widetilde{L}$ and Remark \ref{product_expectation_remark}. Thus $\widetilde{L}$ is a cylindrical stochastic process. To prove that it is in fact a cylindrical L\'evy process, let us fix $n \in \mathbb{N}$ and $g_1,\ldots ,g_n \in G$ and consider the $n$-dimensional processes
$Y$ and $Z$ defined by $Y(t)=(\widetilde{L}(t)g_1,\ldots,\widetilde{L}(t)g_n)$ and $Z(t)=(L(t)g_1,\ldots, L(t)g_n)$.
It is enough to show that for any $m \in \mathbb{N}$ and times $0\leq t_0 <\cdots <t_m\leq T$ the random variables $Y(t_m)-Y(t_{m-1}), \ldots, Y(t_1)-Y(t_0)$ and $Z(t_m)-Z(t_{m-1}), \dots ,Z(t_1)-Z(t_0)$
have the same distribution. Here we only prove that for any $0\leq s<t \leq T$ the random variables $Y(t)-Y(s)$ and $Z(t)-Z(s)$ have the same distribution. The general case follows analogously. To see this, let $A \in \Borel(\mathbb{R}^n)$ be arbitrary. The very definition of $\widetilde{L}$ shows 
\begin{align*}
    &(P \otimes P')\left(Y(t)-Y(s)\in A\right)\\
    &\qquad =(P \otimes P')\left((\widetilde{L}(t)g_1-\widetilde{L}(s)g_1,...,\widetilde{L}(t)g_n-\widetilde{L}(s)g_n)\in A\right)\\
    &\qquad =(P \otimes P')\left(\Omega \times \left\{(L(t)g_1-L(s)g_1,...,L(t)g_n-L(s)g_n)\in A\right\}\right)\\
    &\qquad=P'\left( (L(t)g_1-L(s)g_1,...,L(t)g_n-L(s)g_n)\in A \right)\\
    &\qquad =P\left(Z(t)-Z(s)\in A\right).
\end{align*}
To show that the Radonified increments of $\tilde{L}$ satisfy Condition (1) of Definition \ref{dec_tang_seq_def},  fix some $\omega \in \Omega$. Then $\Theta_n(\omega)$ is a (deterministic) Hilbert-Schmidt operator and $(\widetilde{L}(t)(\omega,\cdot):t \geq 0)$ is a cylindrical Lévy process in $G$. Thus, for a fixed $\omega \in \Omega$ and $n\in \{1,...,N\}$, the Radonified increment $\Theta_n(\omega) (\widetilde{L}(t_n)(\omega,\cdot)-\widetilde{L}(t_{n-1})(\omega,\cdot))$ is an $\mathcal{F}'_{t_n}$-measurable $H$-valued random variable on $(\Omega',\mathcal{F}',P')$ independent of $\mathcal{F}'_{t_{n-1}}$. It follows for each $\omega \in \Omega$ that
\[\left(\Theta_n(\omega) (\widetilde{L}(t_n)(\omega,\cdot)-\widetilde{L}(t_{n-1})(\omega,\cdot))\right)_{n \in \{1,...,N\}}\]
is a sequence of independent random variables. 

For establishing Condition (2) of Definition \ref{dec_tang_seq_def}, we define for each $n \in \{1,...,N\}$
the $H$-valued random variables 
\begin{align*}
&X_n:=\Theta_n\big(L(t_n)-L(t_{n-1})\big):= \sum_{k=1}^{M(n)} \mathbb{1}_{A_{n,k}} \Phi_{n,k} \big(L(t_n)-L(t_{n-1})\big),\\
&Y_n:=\Theta_n\big(\widetilde{L}(t_n)-\widetilde{L}(t_{n-1})\big)
:=\sum_{k=1}^{M(n)} \mathbb{1}_{A_{n,k}} \Phi_{n,k} \big(\widetilde{L}(t_n)-\widetilde{L}(t_{n-1})\big),
\end{align*}
where $F_{n,k} \big(L(t_n)-L(t_{n-1})\big)$ and $ F_{n,k} \big(\widetilde{L}(t_n)-\widetilde{L}(t_{n-1})\big)$ refer to the Radonified increments. Choose regular versions of the conditional distributions
\begin{align*}
    &(P \otimes P')_{X_n}\colon \Borel(H) \times (\Omega \times \Omega') \rightarrow [0,1],\\
    &\qquad\qquad (P \otimes P')_{X_n}\big(B,(\omega,\omega')\big)=(P \otimes P')(X_n \in B \vert \mathcal{F}_{t_{n-1}} \otimes \mathcal{F'}_{t_{n-1}})(\omega,\omega'),\\
    & (P \otimes P')_{Y_n}\colon \Borel(H) \times (\Omega \times \Omega') \rightarrow [0,1], \\ 
     &\qquad\qquad  (P \otimes P')_{Y_n}\big(B,(\omega,\omega')\big)=(P \otimes P')(Y_n \in B \vert \mathcal{F}_{t_{n-1}} \otimes \mathcal{F'}_{t_{n-1}})(\omega,\omega').
 \end{align*}
Since  $\widetilde{L}(t)$ is a cylindrical L\'evy process, we obtain for all $h \in H$ and $n\in{\mathbb N}$ that
\begin{align} \label{EQNU}
     E_{P \otimes P'}&\Big[e^{i\langle Y_n, h\rangle}\Big\vert \mathcal{F}_{t_{n-1}} \otimes \mathcal{F'}_{t_{n-1}}\Big]\nonumber\\
     &=E_{P \otimes P'}\left[e^{i \left \langle \left(\sum_{k=1}^{M(n)}\mathbb{1}_{A_{n,k} \times \Omega'}\Phi_{n,k}\right)(\widetilde{L}(t_n)-\widetilde{L}(t_{n-1})), h \right\rangle}\Big\vert \mathcal{F}_{t_{n-1}} \otimes \mathcal{F'}_{t_{n-1}}\right]\nonumber\\
     &=\sum_{k=1}^{M(n)} E_{P \otimes P'}\Big[\mathbb{1}_{{A_{n,k}} \times \Omega'}\,e^{{i\langle \Phi_{n,k}(\widetilde{L}(t_n)-\widetilde{L}(t_{n-1}))}, h\rangle}\Big\vert \mathcal{F}_{t_{n-1}} \otimes \mathcal{F'}_{t_{n-1}}\Big]\nonumber\\
     &=\sum_{k=1}^{M(n)} \mathbb{1}_{{A_{n,k}} \times \Omega'}\, E_{P \otimes P'}\Big[e^{{i\langle \Phi_{n,k}(\widetilde{L}(t_n)-\widetilde{L}(t_{n-1}))}, h\rangle}\Big\vert \mathcal{F}_{t_{n-1}} \otimes \mathcal{F'}_{t_{n-1}}\Big]\nonumber\\
     &=\sum_{k=1}^{M(n)}\mathbb{1}_{{A_{n,k}} \times \Omega'}\, E_{P \otimes P'}\Big[e^{{i\langle \Phi_{n,k}(\widetilde{L}(t_n)-\widetilde{L}(t_{n-1}))}, h\rangle}\Big]\nonumber\\
     &=\sum_{k=1}^{M(n)}\mathbb{1}_{{A_{n,k}} \times \Omega'}\, E_{P'}\Big[e^{{i\langle \Phi_{n,k}({L}(t_n)-{L}(t_{n-1}))}, h\rangle}\Big]\nonumber\\
     &=\sum_{k=1}^{M(n)}\mathbb{1}_{{A_{n,k}}\times \Omega'}\,e^{{-(t_n-t_{n-1}) \norm{\Phi_{n,k}^*h}^{\alpha}}}\nonumber\\
    &= e^{{-(t_n-t_{n-1}) \norm{\Theta_n^* h}^{\alpha}}}.
\end{align}
In the same way we obtain 
\begin{align} \label{EQNT}
	E_{P \otimes P'}&\Big[e^{i\langle X_n, h\rangle}\Big\vert \mathcal{F}_{t_{n-1}} \otimes \mathcal{F'}_{t_{n-1}}\Big]
	= e^{{-(t_n-t_{n-1}) \norm{\Theta_n^* h}^{\alpha}}}.
\end{align}
It follows from  \eqref{EQNU} and \eqref{EQNT} by calculating the conditional expectation from the conditional probability, see e.g. \cite[Th.\ 6.4]{olav_kallenberg_2002},  that for $P \otimes P'$ a.a.\ $(\omega,\omega') \in \Omega \times \Omega'$ and for all $u\in H$ we have
\begin{align*}
    \varphi_{(P \otimes P')_{X_n}(\cdot,(\omega,\omega'))}(u)&=\int_H e^{i\langle h,u\rangle}\, (P \otimes P')_{X_n}\big({\rm d}h,(\omega,\omega')\big)\\
    &=E_{P \otimes P'}\Big[e^{i\langle X_n, u\rangle}\Big\vert \mathcal{F}_{t_{n-1}} \otimes \mathcal{F'}_{t_{n-1}}\Big](\omega,\omega')\\
    &=E_{P \otimes P'}\Big[e^{i\langle Y_n, u\rangle}\Big\vert \mathcal{F}_{t_{n-1}} \otimes \mathcal{F'}_{t_{n-1}}\Big](\omega,\omega')\\
    &=\int_H e^{i\langle h,u\rangle}\,(P \otimes P')_{X_n}\big({\rm d}h,(\omega,\omega')\big)
    =\varphi_{(P \otimes P')_{Y_n}(\cdot,(\omega,\omega'))}(u).
\end{align*}
Since characteristic functions uniquely determine distributions on $\Borel(H)$, we obtain
\[(P \otimes P')_{X_n}(\cdot,(\omega,\omega'))=(P \otimes P')_{Y_n}(\cdot,(\omega,\omega')) \quad P \otimes P'-\text{a.s.}, \]
establishing Condition (2) of Definition \ref{dec_tang_seq_def}. 
\end{proof}

\section{Characterisation of random integrable processes}

The following is the main result of our work characterising the largest space of predictable integrands which are stochastically integrable with respect to a canonical $\alpha$-stable cylindrical L\'evy process. 
\begin{theorem} \label{pred_iff_integrable}
The space $\mathcal{I}_{\rm prd}^{\rm HS}$ of  predicable Hilbert-Schmidt operator-valued processes integrable 
with respect to a canonical $\alpha$-stable cylindrical L\'evy process  in $G$ for $\alpha\in (0,2)$ 
 coincides with predictable processes in $L_P^0\big(\Omega,L_{{\rm Leb}}^{\alpha}\big([0,T],L_2(G,H)\big)\big)$. 	
\end{theorem}
As in the case of deterministic integrands, the above characterisation of the space of integrable predictable processes 
strongly relies on the equivalent notion of convergences in the two spaces. 
\begin{lemma} \label{le.small-if-small}
	Let $L$ be the canonical $\alpha$-stable cylindrical L\'evy process  in $G$
	and $(\Psi_n)_{n\in{\mathbb N}}$ a sequence in $\mathcal{S}_{\rm adp}^{\rm HS}$. Then the following are 
	equivalent:
	\begin{enumerate}
		\item[{\rm (a)}] $\displaystyle \lim_{n\to\infty}\normm{\Psi_n}_{L^{\alpha}}=0$;
		\item[{\rm (b)}] $\displaystyle \lim_{n\to\infty} \displaystyle\sup_{\Gamma \in \mathcal{S}_{{\rm adp}}^{1, {\rm op}}}E\Bigg[\norm{\int_0^T \Gamma \Psi_n \;{\rm d}L}\wedge 1 \Bigg]=0$.
	\end{enumerate}
\end{lemma}

\begin{proof}
To prove (a) $\Rightarrow$ (b), let $\epsilon>0$ be fixed. Lemma \ref{le.det-small-if-small} enables us to choose 
$\delta>0$ such that for every $\psi\in\mathcal{S}_{\rm det}^{\rm HS}$ we have 
the implication:
\begin{align}\label{eq.det-small-if-small}
	\norm{\psi}_{L^{\alpha}}\le \delta \;\Rightarrow\; 
 \sup_{\gamma \in \mathcal{S}^{1,{\rm op}}_{\rm det}}E\Bigg[\norm{\int_0^T \gamma \psi  \;{\rm d}L}\wedge 1 \Bigg]\le 
 \epsilon. 	
\end{align}
Since $\lim_{n \rightarrow \infty}\normm{\Psi_n}_{L^{\alpha}}=0$, there exists $n_0\in {\mathbb N}$ such that the set $A_n:=\{\norm{\Psi_n}_{L^\alpha}\le \delta\}$ satisfies $P(A_n)\geq 1-\epsilon$ for all $n\ge n_0$. Implication 
\eqref{eq.det-small-if-small} implies for all $\omega\in A_n$ and  $n\geq n_0$ that 
\[
\sup_{\Gamma \in \mathcal{S}_{{\rm adp}}^{1, {\rm op}}}
P'\left(\omega' \in \Omega'\colon \norm{\left(\int_0^T \Gamma(\omega)\Psi_n(\omega) \;{\rm d}\widetilde{L}(\omega,\cdot)\right)(\omega^\prime)}> \epsilon \right)\leq\epsilon.\]
Since $P(A_n)\geq 1- \epsilon $ for all $n \geq n_0$, we obtain 
\begin{align*}
&	\sup_{\Gamma \in \mathcal{S}_{{\rm adp}}^{1, {\rm op}}}  
  P\left( \omega \in \Omega\colon   P'\left(  
  \omega' \in \Omega'\colon  \norm{\left( \int_0^T \Gamma(\omega)\Psi_n(\omega) \, \;{\rm d}\widetilde{L}(\omega,\cdot)\right) (\omega')}> \epsilon
  \right)\leq \epsilon \right)\\
&\qquad\qquad  \geq P(A_n)\geq  1-\epsilon.
\end{align*}
Fubini's theorem implies for all $n\geq n_0$ and $\Gamma \in \mathcal{S}_{{\rm adp}}^{1, {\rm op}}$ that
\begin{align*}
    &(P\otimes P')\left((\omega, \omega')\in \Omega\times \Omega'\colon   \norm{\left(\int_0^T \Gamma \Psi_n \;{\rm d}\widetilde{L}\right)(\omega,\omega')}> \epsilon\right)\\
    &=\int_{\Omega} P'\left(\omega'\in \Omega':  \norm{\left(\int_0^T \Gamma(\omega) \Psi_n(\omega) \;{\rm d}\widetilde{L}(\omega,\cdot)\right)(\omega')}> \epsilon\right)\,P({\rm d}\omega)
    \leq \epsilon + \epsilon(1-\epsilon).
\end{align*}
As $\epsilon>0$ is arbitrary, we obtain
\begin{equation}\label{limit_sup}
    \lim_{n \rightarrow \infty}\sup_{\Gamma \in \mathcal{S}_{{\rm adp}}^{1, {\rm op}}}E_{P \otimes P'} \Bigg[\norm{\int_0^T \Gamma \Psi_n \;{\rm d}\widetilde{L}} \wedge 1\Bigg] = 0.
\end{equation}
By the ideal property of $L_2(G,H)$, for each $n \in \mathbb{N}$ the integrand $\Gamma \Psi_n$ lies in $\mathcal{S}_{\rm adp}^{\rm HS}$ and has a representation of the from
\begin{align}\label{eq.presentiation-GammaS}
\Gamma \Psi_n=\Gamma_0^n\Phi_0^n \mathbb{1}_{\{0\}}+\sum_{i=1}^{N(n)-1} \Gamma_i^n\Phi_i^n \mathbb{1}_{(t_i^n,t_{i+1}^n]},
\end{align}
where $0=t_1^n\leq \cdots <t_{N(n)}^n= T$, and  $\Gamma_i^n\Phi_i^n$ is an $\mathcal{F}_{t_i^n}$-measurable $L_2(G,H)$-valued random variable taking only finitely many values  for each $i = 0,...,N(n)-1$,.
Proposition \ref{dec_tan_seq} guarantees for each $n \in \mathbb{N}$ that the sequence of Radonified increments
\[\Big(\Gamma_i^n\Phi_i^n(L(t_{i+1}^n)-L(t_i^n))\Big)_{i=1,...,N_n-1}\]
has the decoupled tangent sequence
\[\Big(\Gamma_i^n\Phi_i^n (\widetilde{L}(t_{i+1}^n)-\widetilde{L}(t_i^n))\Big)_{i=1,...,N_n-1}.\]
We conclude from  the decoupling inequality \cite[Pr.\ 5.7.1.(ii)]{kwapien_woyczynski_1992} that there exists a constant $c>0$ such that, for all $n \in \mathbb{N}$ and $\Gamma \in \mathcal{S}_{{\rm adp}}^{1, {\rm op}}$, we have
\begin{align}
E_{P\otimes P'}\Bigg[\norm{\int_0^T \Gamma \Psi_n \;{\rm d}L}\wedge 1\Bigg]
&= E_{P\otimes P'}\Bigg[\norm{\sum_{i=1}^{N(n)-1} \Gamma_i^n\Phi_i^n (L(t_{i+1}^n)-L(t_i^n))}\wedge 1\Bigg] \nonumber\\
& \leq c E_{P\otimes P'}\Bigg[\norm{\sum_{i=1}^{N(n)-1} \Gamma_i^n\Phi_i^n (\widetilde{L}(t_{i+1}^n)-\widetilde{L}(t_i^n))} \wedge 1\Bigg]\nonumber \\
&=c E_{P\otimes P'}\Bigg[\norm{\int_0^T \Gamma \Psi_n \;{\rm d}\widetilde{L}}\wedge 1\Bigg].\label{decoupling_ineq}
\end{align}
We conclude from Remark \ref{product_expectation_remark} and \eqref{limit_sup} that 
\begin{equation*} 
 \lim_{n \rightarrow \infty}\sup_{\Gamma \in \mathcal{S}_{{\rm adp}}^{1, {\rm op}}} E_{P}\Bigg[\norm{\int_0^T \Gamma \Psi_n \;{\rm d}L}\wedge 1\Bigg]= \lim_{n \rightarrow \infty}\sup_{\Gamma \in \mathcal{S}_{{\rm adp}}^{1, {\rm op}}} E_{P\otimes P'}\Bigg[\norm{\int_0^T \Gamma \Psi_n \;{\rm d}L}\wedge 1\Bigg]=0, 
\end{equation*}
which shows (b). 

For establishing (b) $\Rightarrow$ (a), we assume the representation \eqref{eq.presentiation-GammaS} of $\Gamma \Psi_n$. We conclude from \cite[Pr.\ 5.7.2]{kwapien_woyczynski_1992} that there exists a constant $c>0$ such that
\begin{align}
    E_{P \otimes P'}\Bigg[\norm{\int_0^T \Gamma \Psi_n \;{\rm d}\widetilde{L}}\wedge 1\Bigg] 
    &= E_{P \otimes P'}\Bigg[\norm{\sum_{i=1}^{N(n)-1}\Gamma_i^n\Phi_i^n(\widetilde{L}(t_{i+1}^n)-\widetilde{L}(t_i^n))}\wedge 1\Bigg] \nonumber\\
    &\leq c \max_{\epsilon_i \in \{\pm1\}} E_{P \otimes P'}\Bigg[\norm{\sum_{i=1}^{N(n)-1}\epsilon_i \Gamma_i^n\Phi_i^n(L(t_{i+1}^n)-L(t_i^n))}\wedge 1\Bigg] \nonumber\\
    &=c \max_{\epsilon_i \in \{\pm1\}} E_{P}\Bigg[\norm{\sum_{i=1}^{N(n)-1}\epsilon_i\Gamma_i^n \Phi_i^n(L(t_{i+1}^n)-L(t_i^n))}\wedge 1\Bigg] \nonumber\\
    &\leq c \sup_{\Gamma \in \mathcal{S}_{{\rm adp}}^{1, {\rm op}}} E_{P}\Bigg[\norm{\sum_{i=1}^{N(n)-1}\Gamma_i^n\Phi_i^n(L(t_{i+1}^n)-L(t_i^n))}\wedge 1\Bigg] \nonumber\\
    &=c \sup_{\Gamma \in \mathcal{S}_{{\rm adp}}^{1, {\rm op}}} E_{P}\Bigg[\norm{\int_0^T \Gamma \Psi_n \;{\rm d}L}\wedge 1\Bigg].\label{main_proof_decoupl_ineq}
\end{align}
By choosing $\Gamma=\text{Id}\mathbb{1}_{\Omega \times (0,T]}$, the hypothesis on $(\Psi_n)_{n\in{\mathbb N}}$ implies 
 \[\lim_{n \rightarrow \infty} E_{P \otimes P'}\Bigg[\norm{\int_0^T \Psi_n \;{\rm d}\widetilde{L}}\wedge 1\Bigg] = 0. \]
It follows that there exists a subsequence $(\Psi_{n_k})_{k \in \mathbb{N}}$ of $(\Psi_n)_{n \in \mathbb{N}}$ and a set $N \subseteq \Omega \times \Omega'$ with $(P \otimes P')(N)=0$ satisfying 
\[\lim_{k \rightarrow \infty}\Bigg(\int_0^T \Psi_{n_k}\;{\rm d}\widetilde{L}\Bigg)(\omega,\omega')=0 \quad \text{for each}\; (\omega,\omega')\in N^c.\]
Define the section of the set $N$ for each $\omega \in \Omega$  by
\[N_{\omega}=\Bigg\{\omega' \in \Omega'\colon  \lim_{k \rightarrow \infty}\left(\int_0^T \Psi_{n_k}(\omega)\;{\rm d} \widetilde{L}(\omega,\cdot)\right)(\omega')\neq 0\Bigg\},\]
where we note that since $\Psi_{n_k}$ are adapted step processes it holds that
\[\Bigg(\int_0^T \Psi_{n_k}\;{\rm d}\widetilde{L}\Bigg)(\omega,\cdot)
=\int_0^T \Psi_{n_k}(\omega) \;{\rm d} \widetilde{L}(\omega,\cdot)\qquad\text{for all }\omega\in\Omega .\]
Fubini's theorem implies $0=(P \otimes P')(N)= \int_{\Omega}P'(N_{\omega}){\rm d}P(\omega)$,  from which it follows that there exists  $\Omega_1 \subseteq \Omega$ with $P(\Omega_1)=1$ such that $P'(N_{\omega})=0$  for all $\omega \in \Omega_1$. In other words, for each fixed $\omega \in \Omega_1$, the sequence of random variables
\[\Bigg(\int_0^T \Psi_{n_k}(\omega) \;{\rm d} \widetilde{L}(\omega,\cdot)\Bigg )_{k \in \mathbb{N}}\]
converges $P'$-a.s.\ to $0$ as $H$-valued random variables on $(\Omega',\mathcal{F}',P')$. Lemma~\ref{le.det-small-if-small} implies $\lim_{n \rightarrow \infty}\norm{\Psi_{n_k}(\omega)}_{L^{\alpha}}=0$  for each fixed $\omega \in \Omega_1$.  As $P(\Omega_1)=1$, Lebesgue's dominated convergence theorem implies
\[\lim_{n \rightarrow \infty}\normm{\Psi_{n_k}}_{L^{\alpha}}=\lim_{n \rightarrow \infty}\int_{\Omega} \Big(\norm{\Psi_{n_k}(\omega)}_{L^{\alpha}} \wedge 1\Big)\,P({\rm d}\omega) =0, \]
which completes the proof.
\end{proof}

\begin{proof} [Proof of Theorem \ref{pred_iff_integrable}.] 
	If $\Psi\in \mathcal{I}_{\rm prd}^{\rm HS}$ then Definition \ref{pred_integrability} guarantees the existence of a sequence $(\Psi_n)_{n \in \mathbb{N}}$ of elements of $\mathcal{S}_{\rm adp}^{\rm HS}$ converging $P_T$-a.e. to $\Psi$ and satisfying 
	\[\displaystyle \lim_{m,n \rightarrow \infty}\sup_{\Gamma \in \mathcal{S}_{{\rm adp}}^{1, {\rm op}}}E\Bigg[\norm{\int_0^T \Gamma(\Psi_m-\Psi_n) \;{\rm d}L}\wedge1 \Bigg]=0.\]
	 Lemma \ref{le.small-if-small} implies that $\lim_{m,n \rightarrow \infty}\normm{\Psi_m-\Psi_n}_{L^{\alpha}}=0$. Completeness of the space $L_P^0\big(\Omega,L_{{\rm Leb}}^{\alpha}\big([0,T],L_2(G,H)\big)\big)$ and the fact that $(\Psi_n)_{n \in \mathbb{N}}$ converges $P_T$-a.e. to $\Psi$ yields that the sequence $(\Psi_n)_{n \in \mathbb{N}}$ has a limit in $L_P^0\big(\Omega,L_{{\rm Leb}}^{\alpha}\big([0,T],L_2(G,H)\big)\big)$ and that this limit necessarily coincides with $\Psi$. Thus $\Psi\in L_P^0\big(\Omega,L_{{\rm Leb}}^{\alpha}\big([0,T],L_2(G,H)\big)\big)$. 
	
	To establish the converse inclusion, let $\Psi$ be a predictable process in the space  $L_P^0\big(\Omega,L_{{\rm Leb}}^{\alpha}\big([0,T],L_2(G,H)\big)\big)$. Lemma \ref{pred_density} guarantees that there exists a sequence $(\Psi_n)_{n \in \mathbb{N}}$ of elements of $\mathcal{S}_{\rm adp}^{\rm HS}$ converging to $\Psi$ in $\normm{\cdot}_{L^{\alpha}}$ and $P_T$-a.e. Consequently, $\lim_{m,n \rightarrow \infty}\normm{\Psi_m-\Psi_n}_{L^{\alpha}}=0$, which implies by Lemma \ref{le.small-if-small} that
	\[\displaystyle \lim_{m,n \rightarrow \infty}\sup_{\Gamma \in \mathcal{S}_{{\rm adp}}^{1, {\rm op}}}E\Bigg[\norm{\int_0^T \Gamma(\Psi_m-\Psi_n) \;{\rm d}L}\wedge1 \Bigg]=0.\]
	Thus $\Psi$ satisfies the conditions of Definition \ref{pred_integrability}, which means that $\Psi \in \mathcal{I}_{\rm prd}^{\rm HS}$.
\end{proof}

Lemma  \ref{le.small-if-small} is crucial to characterise the space of integrable adapted processes in Theorem \ref{pred_iff_integrable}, as it describes convergence of adapted step processes in the space of integrands in terms of convergence in the explicitly given space $L^\alpha$.  Having identified the space of integrable adapted processes, we can extend Lemma \ref{le.small-if-small} to the general class of integrable processes. 
\begin{corollary} \label{co.small-if-small}
	Let $L$ be the canonical $\alpha$-stable cylindrical L\'evy process  in $G$
	and $(\Psi_n)_{n\in{\mathbb N}}$ a sequence in $\mathcal{I}_{\rm prd}^{\rm HS}$. Then the following are 
	equivalent:
	\begin{enumerate}
		\item[{\rm (a)}] $\displaystyle \lim_{n\to\infty}\normm{\Psi_n}_{L^{\alpha}}=0$; 
		\item[{\rm (b)}] $\displaystyle \lim_{n\to\infty} \displaystyle\sup_{\Gamma \in \mathcal{S}_{{\rm adp}}^{1, {\rm op}}}E\Bigg[\norm{\int_0^T \Gamma \Psi_n \;{\rm d}L}\wedge 1 \Bigg]=0$. 
	\end{enumerate}
\end{corollary}
\begin{proof}
 To establish the implication (a) $\Rightarrow$ (b) let $\epsilon>0$ be fixed. Lemma \ref{le.small-if-small} implies that there exists a $\delta(\epsilon)>0$ such that we have for all $\Psi \in \mathcal{S}_{\rm adp}^{\rm HS}$ the implication:
 \begin{equation}\label{epsilon_delta_implication}
     \normm{\Psi}_{L^{\alpha}}<\delta(\epsilon) \quad \Rightarrow \quad \displaystyle\sup_{\Gamma \in \mathcal{S}_{{\rm adp}}^{1, {\rm op}}}E\Bigg[\norm{\int_0^T \Gamma \Psi \;{\rm d}L}\wedge 1 \Bigg]<\epsilon.
 \end{equation}
 Since $\lim_{n\to\infty}\normm{\Psi_n}_{L^{\alpha}}=0$, there exists an $n_0 \in \mathbb{N}$ such that $\normm{\Psi_n}_{L^{\alpha}}<\tfrac{\delta(\epsilon)}{2}$  for all $n \geq n_0$. As $(\Psi_n)_{n \in \mathbb{N}} \subseteq L_P^0\big(\Omega,L_{{\rm Leb}}^{\alpha}\big([0,T],L_2(G,H)\big)\big)$,  Lemma \ref{pred_density} guarantees for each $n \in \mathbb{N}$ the existence of a sequence $(\Psi_n^m)_{m \in \mathbb{N}}\subseteq \mathcal{S}_{\rm adp}^{\rm HS}$ converging to $\Psi_n$ in $\normm{\cdot}_{L^{\alpha}}$ and $P_T$-a.e. Consequently, we can find $m_0(n,\epsilon) \in \mathbb{N}$ for each $n \in \mathbb{N}$ such that for all $m \geq m_0(n,\epsilon)$ we have 
$\normm{\Psi_n^m-\Psi_n}_{L^{\alpha}}< \tfrac{\delta(\epsilon)}{2}$.
We obtain for each $n\geq n_0$ and $m \geq m_0(n,\epsilon)$ that
 \[\normm{\Psi_n^m}_{L^{\alpha}}\leq \normm{\Psi_n^m-\Psi_n}_{L^{\alpha}}+\normm{\Psi_n}_{L^{\alpha}}< \delta(\epsilon),\]
which implies by (\ref{epsilon_delta_implication}) that
 \begin{equation}\label{eq_small_approx_int}
     \displaystyle\sup_{\Gamma \in \mathcal{S}_{{\rm adp}}^{1, {\rm op}}}E\Bigg[\norm{\int_0^T \Gamma \Psi_n^m \;{\rm d}L}\wedge 1 \Bigg]<\epsilon.
 \end{equation}
Thus, if we fix an $n\geq n_0$ and recall that the integral of $\Psi_n$ is defined to be the limit in probability of the integrals of $\Psi_n^m$ as $m\to\infty$, we obtain from Equation (\ref{eq_small_approx_int}) that
\begin{align*}
         \displaystyle\sup_{\Gamma \in \mathcal{S}_{{\rm adp}}^{1, {\rm op}}}E\Bigg[\norm{\int_0^T \Gamma \Psi_n \;{\rm d}L}\wedge 1 \Bigg] &= \sup_{\Gamma \in \mathcal{S}_{{\rm adp}}^{1, {\rm op}}} \lim_{m \rightarrow \infty}E\Bigg[\norm{\int_0^T \Gamma \Psi_n^m \;{\rm d}L}\wedge 1 \Bigg] \\
         &\leq \lim_{m \rightarrow \infty} \sup_{\Gamma \in \mathcal{S}_{{\rm adp}}^{1, {\rm op}}}E\Bigg[\norm{\int_0^T \Gamma \Psi_n^m \;{\rm d}L}\wedge 1 \Bigg]< \epsilon.
\end{align*}
 
To establish the reverse implication (b) $\Rightarrow$ (a), let $\epsilon>0$ be fixed. Lemma \ref{le.small-if-small} implies that there exists a $\delta(\epsilon)>0$ such that we have for all $\Psi \in \mathcal{S}_{\rm adp}^{\rm HS}$ the implication: 
 \begin{equation}\label{reverse_epsilon_delta_implication}
     \displaystyle\sup_{\Gamma \in \mathcal{S}_{{\rm adp}}^{1, {\rm op}}}E\Bigg[\norm{\int_0^T \Gamma \Psi \;{\rm d}L}\wedge 1 \Bigg]<\delta(\epsilon) \quad \Rightarrow \quad \normm{\Psi}_{L^{\alpha}}<\tfrac{\epsilon}{2}.
 \end{equation}
 By assumption, there exists an $n_0 \in \mathbb{N}$ such that for all $n \geq n_0$ we have
 \begin{equation} \label{eq.small_integral}
      \displaystyle\sup_{\Gamma \in \mathcal{S}_{{\rm adp}}^{1, {\rm op}}}E\Bigg[\norm{\int_0^T \Gamma \Psi_n \;{\rm d}L}\wedge 1 \Bigg]<\tfrac{\delta(\epsilon)}{2}.
 \end{equation}
 As $(\Psi_n)_{n \in \mathbb{N}} \subseteq \mathcal{I}_{\rm prd}^{\rm HS}$, it follows from Theorem \ref{pred_iff_integrable} and Lemma \ref{pred_density} that for each $n \in \mathbb{N}$  there exists a sequence $(\Psi_n^m)_{m \in \mathbb{N}}$ of elements of $\mathcal{S}_{\rm adp}^{\rm HS}$ converging to $\Psi_n$ in $\normm{\cdot}_{L^{\alpha}}$ and $P_T$-a.e. Consequently,  we can find $m_0(n,\epsilon) \in \mathbb{N}$  for each $n \in \mathbb{N}$, such that for all $m \geq m_0(n,\epsilon)$ we have
 \begin{equation} \label{eq.process_dif_is_small}
     \normm{\Psi_n^m-\Psi_n}_{L^{\alpha}}< \epsilon/2
 \end{equation}
Definition \ref{pred_integrability} shows that for each $n \in \mathbb{N}$ there exists an $m_1(n,\epsilon) \in \mathbb{N}$ such that for all $m \geq m_1(n,\epsilon)$ we have by the reverse triangle inequality that
 \begin{equation}\label{eq.integrals_close}
     \left \vert\displaystyle\sup_{\Gamma \in \mathcal{S}_{{\rm adp}}^{1, {\rm op}}}E\Bigg[\norm{\int_0^T \Gamma \Psi_n \;{\rm d}L}\wedge 1 \Bigg]-\displaystyle\sup_{\Gamma \in \mathcal{S}_{{\rm adp}}^{1, {\rm op}}}E\Bigg[\norm{\int_0^T \Gamma \Psi_n^m \;{\rm d}L}\wedge 1 \Bigg]\right\vert<\tfrac{\delta(\epsilon)}{2}.
 \end{equation}
Combining (\ref{eq.small_integral}) and (\ref{eq.integrals_close}) shows for  $n\geq n_0$ and $m \geq \max\{m_0(n,\epsilon),m_1(n,\epsilon)\}$ that
 \[\sup_{\Gamma \in \mathcal{S}_{{\rm adp}}^{1, {\rm op}}}E\Bigg[\norm{\int_0^T \Gamma \Psi_n^m \;{\rm d}L}\wedge 1 \Bigg]<\delta(\epsilon),\]
 which implies by (\ref{reverse_epsilon_delta_implication}) and (\ref{eq.process_dif_is_small}) that
 \[\normm{\Psi_n}_{L^{\alpha}}\leq \normm{\Psi_n-\Psi_n^m}_{L^{\alpha}}+\normm{\Psi_n^m}_{L^{\alpha}}<\epsilon. \]
 As $\epsilon>0$ was arbitrary, this concludes the proof.
\end{proof}

\begin{theorem}\label{th_semimartingale}
If $\Psi \in \mathcal{I}_{\rm prd}^{\rm HS}$, then the integral process $\left(I(\Psi)(t): t \in [0,T]\right)$ defined  by
\begin{align*}
    I(\Psi)(t):=\int_0^T \1_{[0,t]}(s) \Psi(s)\, L({\rm d} s)
    \qquad\text{for }t\in [0,T], 
\end{align*}
is a semi-martingale.
\end{theorem}

\begin{proof}
By \cite[Th.\ 2.1]{radonif_by_single}, it suffices to show that the set
\begin{align*}
    \left\{ \int_0^T \Gamma \, {\rm d}I(\Psi) : \Gamma \in \mathcal{S}_{\rm adp}^{1,\rm{op}} \right\}
\end{align*}
is bounded in probability. Suppose, aiming for a contradiction, that it is not the case. Then there exists an $\epsilon>0$ and a sequence $(\Gamma_n)_{n \in \mathbb{N}}\subseteq \mathcal{S}_{\rm adp}^{1,\rm{op}}$ satisfying for all $n \in \mathbb{N}$ that
\begin{align}\label{eq.contradiction}
    P\left(\norm{\int_0^T \Gamma_n \, {\rm d}I(\Psi)}>n\right)\geq \epsilon.
\end{align}
For each $\Psi \in \mathcal{S}_{\rm adp}^{\rm HS}$ and $\Gamma \in \mathcal{S}_{\rm adp}^{1,\rm{op}}$, the very definitions of stochastic integrals show 
\begin{equation*}
    \int_0^T \Gamma \, {\rm d}I(\Psi) = \int_0^T \Gamma \Psi \, {\rm d}L. 
\end{equation*}
This equality can be generalised to arbitrary $\Psi \in \mathcal{I}_{\rm prd}^{\rm HS}$ and $\Gamma \in \mathcal{S}_{\rm adp}^{1,\rm{op}}$ by a standard approximation argument. Using this to rewrite Equation (\ref{eq.contradiction}), we obtain for all $n \in \mathbb{N}$ that
\begin{align}\label{eq.bounded_set}
    \epsilon \leq P\left(\norm{\int_0^T \Gamma_n \, {\rm d}I(\Psi)}>n\right) 
    = P\left(\norm{\int_0^T \frac{1}{n}\Gamma_n \Psi \, {\rm d}L}>1\right).
\end{align}
On the other hand, since $\normm{\frac{1}{n}\Gamma_n \Psi}_{L^{\alpha}} \rightarrow 0$ as $n \rightarrow \infty$,  Corollary \ref{co.small-if-small} implies
\begin{align*}
    \lim_{n \rightarrow \infty}E\left[ \norm{\int_0^T \frac{1}{n}\Gamma_n \Psi \, {\rm d}L} \wedge 1 \right]=0,
\end{align*}
which contradicts (\ref{eq.bounded_set}) because of the equivalent characterisation of the topology in $L^0_P(\Omega,H)$.
\end{proof}

We finish this section with a stochastic dominated convergence theorem.

\begin{theorem} \label{stoch_dom_conv}
	Let $(\Psi_n)_{n \in \mathbb{N}}$ be a sequence of processes in $\mathcal{I}_{\rm prd}^{\rm HS}$ such that
	\begin{enumerate}[\rm (1)]
		\item \label{thm_cnd1'} $(\Psi_n)_{n \in \mathbb{N}}$ converges $P_T$-a.e.\ to an $L_2(G,H)$-valued predictable process $\Psi$;
		\item \label{thm_cnd2'} there exists a process $\Upsilon \in \mathcal{I}_{\rm prd}^{\rm HS}$ satisfying for all $n \in \mathbb{N}$ that
		\[\norm{\Psi_n(\omega,t)}_{\rm HS} \leq \norm{\Upsilon(\omega,t)}_{\rm HS} \quad  \text{for $P_T$-a.a.\ $(\omega,t)\in \Omega\times [0,T]$.} \]
	\end{enumerate}
	Then it follows that $\Psi \in \mathcal{I}_{\rm prd}^{\rm HS}$ and 
	\[\lim_{n \rightarrow \infty}P\Bigg(\sup_{t \in [0,T]}\norm{\int_0^t \Psi_n \;{\rm d}L-\int_0^t \Psi \;{\rm d}L}>\epsilon\Bigg)=0
	\qquad\text{for all } \epsilon>0.\]
\end{theorem}

\begin{proof}
By assumption, there exists a set $N \subseteq \Omega \times [0,T]$ with $P_T(N)=0$ such that  $\lim_{n \rightarrow \infty} \Psi_n(\omega,t)=\Psi(\omega,t)$ and  $\norm{\Psi_n(\omega,t)}_{\rm HS} \leq \norm{\Upsilon(\omega,t)}_{\rm HS}$ for all $(\omega,t)\in N^c$ and $n \in \mathbb{N}$. Fubini's theorem yields that
\[0=P_T(N)=\int_\Omega {\rm Leb}\vert_{[0,T]}(N_\omega)\,P({\rm d}\omega),\]
where
\[N_\omega:=\left\{t\in [0,T]: \lim_{n \rightarrow \infty} \Psi_n(\omega,t)\neq\Psi(\omega,t)\}\,\text{or}\, \norm{\Psi_n(\omega,t)}_{\rm HS} > \norm{\Upsilon(\omega,t)}_{\rm HS}\right\}.\]
It follows that there exists an $\Omega_1 \subseteq \Omega$ with $P(\Omega_1)=1$ such that ${\rm Leb}\vert_{[0,T]}(N_{\omega})=0$ for all $\omega \in \Omega_1$. Consequently, for each $\omega \in \Omega_1$ we have $\norm{\Psi_n(\omega,t)}_{\rm HS}\leq \norm{\Upsilon(\omega,t)}_{\rm HS}$ and $\lim_{n \rightarrow \infty} \Psi_n(\omega,t)=\Psi(\omega,t)$ for Lebesgue almost every $t \in [0,T]$.  Theorem \ref{pred_iff_integrable} guarantees 
that  there exists $\Omega_2 \subseteq \Omega$ with $P(\Omega_2)=1$ such that $\norm{\Upsilon(\omega,\cdot)}_{L^{\alpha}}<\infty$  for all $\omega \in \Omega_2$. The classical Lebesgue's dominated convergence theorem implies that $(\Psi_n(\omega))_{n \in \mathbb{N}}$ converges in $\norm{\cdot}_{L^{\alpha}}$  and its limit  $\Psi(\omega,\cdot)$ is in $L^\alpha$ for all $\omega \in \Omega_1 \cap \Omega_2$.  Since $P(\Omega_1 \cap \Omega_2)=1$, Theorem \ref{pred_iff_integrable} shows $\Psi \in \mathcal{I}_{\rm prd}^{\rm HS}$. 

Another application of Lebesgue's dominated convergence theorem establishes that  $\lim_{n \rightarrow \infty}\normm{\Psi_n-\Psi}_{L^{\alpha}}=0$.  Corollary \ref{co.small-if-small} implies
\begin{align} \label{eq_conv_to_0}
    \displaystyle \lim_{n\to\infty} \displaystyle\sup_{\Gamma \in \mathcal{S}_{{\rm adp}}^{1, {\rm op}}}E\Bigg[\norm{\int_0^T \Gamma (\Psi_n-\Psi) \;{\rm d}L}\wedge 1 \Bigg]=0,
\end{align}
which means that the sequence $\left(I(\Psi_n)\right)_{n \in \mathbb{N}}$ of processes converges in the semi-martingale topology to the process $I(\Psi)$. Since convergence in the semi-martingale topology implies convergence in probability on compact time intervals, we have
\[\lim_{n \rightarrow \infty}P\Bigg(\sup_{t \in [0,T]}\norm{\int_0^t \Psi_n \;{\rm d}L-\int_0^t \Psi \;{\rm d}L}>\epsilon\Bigg)=0
\qquad\text{for all } \epsilon>0,\]
which completes the proof.
\end{proof}

\end{document}